\documentclass[11pt, leqno]{amsart}

\usepackage{amsfonts, amsmath, 
}

\allowdisplaybreaks
\usepackage{tikz}	
\usepackage{mathtools}
\usepackage{esint}
\usepackage[margin=1.4in]{geometry}

\usepackage{amssymb,amsthm,
paralist
}

\usepackage{
latexsym,
}

\usepackage{enumitem}

\definecolor{darkgreen}{rgb}{0,0.5,0}
\definecolor{darkred}{rgb}{0.7,0,0}
\usepackage[colorlinks, 
citecolor=darkgreen, linkcolor=darkred
]{hyperref}

\theoremstyle{plain}
\newtheorem{lemma}{Lemma}[section]
\newtheorem{thm}[lemma]{Theorem}

\newtheorem{cor}[lemma]{Corollary}

\newtheorem{pb}[lemma]{Problem}

\theoremstyle{definition}
\newtheorem{defn}[lemma]{Definition}

\newtheorem{rmk}[lemma]{Remark}

\numberwithin{equation}{section}

\newcommand{\of}{{\circ}}

\newcommand{\boundary}{\partial}

\newcommand{\moll}[1]{\hat{  #1}}

\newcommand{\partt}{ {\frac{\partial}{\partial t} } }

\renewcommand{\phi}{\varphi}
\newcommand{\ti}{\tilde}
\newcommand{\parts}{\frac{\partial}{\partial s} }
\newcommand{\al}{\alpha}
\newcommand{\be}{\beta}
\newcommand{\ga}{\gamma}

\newcommand{\de}{\delta}

\newcommand{\la}{\lambda}

\newcommand{\si}{\sigma}

\newcommand{\ep}{\varepsilon}

\newcommand{\curlK}{\mathcal K}
\newcommand{\curlX}{\mathcal X}
\newcommand{\curlR}{\mathcal R}
\newcommand{\Cone}{\mathcal C}
\newcommand{\Id}{Id}
\newcommand{\R}{\ensuremath{{\mathbb R}}}
\newcommand{\Sp}{\ensuremath{{\mathbb S}}}
\newcommand{\N}{\ensuremath{{\mathbb N}}}
\newcommand{\B}{\ensuremath{{\mathbb B}}}

\newcommand{\T}{\ensuremath{{\mathbb T}}}

\newcommand{\downto}{\searrow}

\newcommand{\lap}{\Delta}

\newcommand{\grad}{\nabla}

\newcommand{\vol}{{ \rm Vol}}

\newcommand{\norm}[1]{\left\Vert#1\right\Vert}  

\DeclareMathOperator{\AVR}{AVR}

\newcommand{\beq}{\begin{equation}}
\newcommand{\eeq}{\end{equation}}
\newcommand{\beqa}{\begin{equation}\begin{aligned}}
\newcommand{\eeqa}{\end{aligned}\end{equation}}
\newcommand{\brmk}{\begin{rmk}}
\newcommand{\ermk}{\end{rmk}}
\newcommand{\partref}[1]{\hbox{(\csname @roman\endcsname{\ref{#1}})}}

\newcommand{\Rm}{{\mathrm{Rm}}}
\newcommand{\Riem}{{\mathrm{Rm}}}
\newcommand{\Rc}{{\mathrm{Ric}}}
\newcommand{\Rcci}{{\mathrm{Ric}}}

\newcommand{\Sc}{{\mathrm{R}}}
\newcommand{\Hess}{{\mathrm{Hess}}}

\usepackage{soul}

\newsavebox\CBox
\newcommand\hcancel[2][0.5pt]{%
  \ifmmode\sbox\CBox{$#2$}\else\sbox\CBox{#2}\fi%
  \makebox[0pt][l]{\usebox\CBox}%
  \rule[0.5\ht\CBox-#1/2]{\wd\CBox}{#1}}

\title{
{
On the regularity of Ricci flows coming out of metric spaces
} 
\\ 
}
\author{Alix Deruelle, 
Felix Schulze and 
Miles Simon}
\date{\today}

\begin{document}

\maketitle
\begin{abstract}\noindent 
We consider smooth, not necessarily complete,  $n$-dimensional Ricci flows, $(M,g(t))_{t\in (0,T)}$  with ${\mathrm{Ric}}(g(t)) \geq -1$  and $| {\mathrm{Rm}}  (g(t))| \leq c/t$ for all $t\in (0 ,T)$ 
 {\it coming out} of metric spaces  $(M,d_0)$ in the sense that
$(M,d(g(t)), x_0) \to (M,d_0, x_0)$ as $t\searrow 0$ in the pointed Gromov-Hausdorff sense. 
In the case that $B_{g(t)}(x_0,1)  \Subset M$ for all $t\in (0,T)$ and 
$(B_{d_0}(x_0,1),d_0)$ can be isometrically and compactly  embedded in a smooth $n$-dimensional Riemannian manifold
$(\Omega,d(\ti g_0))$, then  we show using Ricci-harmonic map heat flow, that there is a corresponding smooth  solution
$\tilde g(t)_{t\in (0,T)}$ to  the $\delta$-Ricci-DeTurck flow on an Euclidean ball ${\mathbb B}_{r}(p_0) \subset {\mathbb R}^n$, for some small $0<r<1,$ and $\ti g(t) \to \ti g_0$ smoothly as $t\to 0$.
We further show that this implies that the original solution $g$ can be extended locally to a smooth solution   defined up to  time zero,  in view of the method of Hamilton.
\end{abstract}

\parskip=0.5pt
\tableofcontents
\parskip=0pt






\section{Introduction}
\subsection{Overview}\label{overview}
In this paper, we investigate and answer in certain cases the following question.
\begin{pb} \label{main-pb}
Let $(M,g(t))_{t\in(0,T)}$ be a (possibly incomplete) Ricci flow which satisfies
\begin{equation}\label{eq:intro.1}
\quad |\Rm(\cdot,t)| \leq \frac{c_0}{t},
\end{equation} 
and for which $(M,d(g(t)))$ Gromov-Hausdorff converges to a metric space $(X,d_0)$ as $t\downto 0$. 

What further assumptions on the regularity of $(X,d_0)$ and $(M,g(t))_{t\in (0,T)}$ 
guarantee that $g(t)$ converges locally smoothly (or continuously) to a smooth (or continuous) metric as $t$ approaches zero?
\end{pb}

\begin{rmk} We recall that for a connected, open, not necessarily complete, Riemannian manifold $(M,g)$, there is a metric $d(g)$ induced by $g$ which makes $(M,d)$ into a metric space. Note that the distance between two points $p,q \in M$ is not necessarily realised by a geodesic, nevertheless for every $x \in M$ there exists $r>0$ such $B_{d(g)}(x,r)$ is geodesically convex, and the distance between any two points in $B_{d(g)}(x,r)$ is uniquely realised by a smooth geodesic.
\end{rmk}

If we assume in Problem \ref{main-pb} that $(M,g(t))_{t\in (0,T)}$ is complete for each $t \in (0,T)$ and in addition to \eqref{eq:intro.1} it holds that
\begin{equation}\label{basic-ass-curv-bd-intro}
\Rc(\cdot,t) \geq -1
\end{equation}
for all $t \in (0,T)$, then $X$ is homeomorphic to $M$ and the topology of $(M, d_0)$ agrees with that of $(M,g(t))$ for all $t \in (0,T)$. This is a consequence of the following estimate on the induced distances, see \cite[Lemma 3.1]{SiTo2}: Let $d_t = d(g(t))$, then $d_t \to d_0$ for a metric $d_0$ on $M$ and 
\begin{equation}
e^{t} d_0 \geq d_t \geq d_0 -\ga(n) \sqrt{c_0 t}  \ \mbox{ for all } \ t \in (0,T ).  \label{easydistest}
\end{equation}
This implies convergence of $d_t$ in the $C^0$-sense to $d_0$, which is stronger than Gromov-Hausdorff convergence. Since Gromov-Hausdorff limits are unique up to isometries, this implies:
if $(M,d(g(t_i)),p) \to (X,d_X,x)$ in the Gromov-Hausdorff sense for a sequence of times $t_i \downto 0$ then $(X,d_X,x)$ is isometric to $(M,d_0,p)$. Hence it is {\it not} possible that complete solutions satisfying \eqref{eq:intro.1} and \eqref{basic-ass-curv-bd-intro} come out of metric spaces which are {\it not} manifolds.

Note that if $\Rc(\cdot,t) \geq -C,$  and \eqref{eq:intro.1} holds for all $t\in (0,T)$  for some $C > 1$, then, since \eqref{eq:intro.1} is invariant under scaling, we can scale the solution such that it satisfies \eqref{basic-ass-curv-bd-intro}.

Estimate \eqref{easydistest} can be localised as follows.

\begin{lemma}(Simon-Topping, \cite[Lemma 3.1]{SiTo2}).\label{SiTopThm1}
Let $(M,g(t))_{t\in (0,T)}$, $T\leq 1$, be a smooth Ricci flow, satisfying $\Rc(\cdot,t) \geq -1 , \quad |\Rm(\cdot,t)| \leq c_0/t$, where 
 $M$ is connected but $(M,g(t))$ not  necessarily complete. Assume furthermore that $B_{g(t)}(x_0,1) \Subset M$ for all $t\in (0,T)$.    
 
Then $X:= (\cap_{s \in (0,T)} B_{g(s)}(x_0,\frac 1 2))$ is non-empty and there is a well defined limiting metric $d_t \to d_0$ as $t \downto 0$, where 
\begin{eqnarray}
&&e^td_ 0  \geq d_t  \geq d_0 -\ga(n)  \sqrt{c_0 t} \ \ \mbox{ for all } t\in [0,T)  \mbox{ on } X.  \label{distestinit} 
\end{eqnarray}

Furthermore, there exists   $R= R(c_0,n)>0, S = S(c_0,n) >0$ such that
$B_{d_0}(x_0,r) \Subset \curlX \subseteq X$  
and $B_{g(t)}(x_0,r) \Subset \curlX \subseteq X $ for all $r \leq R(c_0,n)$ and $t \leq S$ where $\curlX$ is the connected component of $X$ which contains $x_0$, and the topology 
of $ B_{d_0}(x_0,r)$ induced by $d_0$ agrees with that of the set $  B_{d_0}(x_0,r)\subseteq M$ induced by the topology of $M$. 

 \end{lemma}

Coming back to our initial question, assuming \eqref{eq:intro.1} and \eqref{basic-ass-curv-bd-intro} hold and for some $r>0$, $B_{g(t)}(x_0,r) \Subset M$  for all $t\in (0,T)$, 
  the above result implies that the metric $d_0$ exists (locally) and the convergence is in the sense of \eqref{distestinit}.\\[2ex]
{\bf Examples.} We give  examples of solutions satisfying the conditions \eqref{eq:intro.1} and \eqref{basic-ass-curv-bd-intro}.\\[-2ex]
\begin{enumerate}
 \item \label{Ex1} {\bf Expanding gradient Ricci solitons coming out of non-ne\-ga\-ti\-ve\-ly curved cones.}\\[0.5ex]
Consider a smooth Riemannian metric $\ga$ on $\Sp^{n-1}$ with eigenvalues of its curvature operator greater or equal to one and the cone $C(\ga) = (([0,\infty) \times \Sp^{n-1})/ \sim ,dr^2 \oplus r^2 \ga, O)$ where the equivalence relation $\sim$ identifies $O:=(0,x) \sim (0,y)$. Note that the curvature operator of $C(\ga)$ is non-negative away from the tip. In \cite{SchulzeSimon} it was shown that if $C(\ga)$ arises as the tangent cone at infinity of a non-compact manifold $M$ with non-negative and bounded curvature operator then there exists an expanding gradient soliton $(M,g)$ such that its evolution under Ricci flow $(M, (g(t))_{t\in (0,\infty)}$ has the property that $(M,d(g(t) ),p) \to C(\ga)$ in the pointed Gromov-Hausdorff sense as $t \searrow 0$. The  construction in \cite{SchulzeSimon} guarantees that this convergence is in $C^{\infty}_{loc}$  away from the tip of the cone. \\
In \cite{DeruelleExpanders} it was later shown that there always exists an expanding gradient soliton coming out of any such non-negatively curved cone $C(\ga)$. The construction in \cite{DeruelleExpanders}  also guarantees that the convergence is in $C^{\infty}_{loc}$ away from the tip: the existence result is based on the Nash-Moser fixed point theorem. Problem \ref{main-pb} was partly motivated by the cost of using such a "black box". Indeed, the Nash-Moser fixed point theorem is not so sensitive to the nature of the non-linearities of the Ricci flow equation as long as the corresponding linearized operator satisfies the appropriate Fredholm properties. In particular, the use of the Nash-Moser fixed point theorem does not shed new light on the smoothing effect of the Ricci flow. Finally, we emphasize  that uniqueness of such solutions is unknown among the class of asymptotically conical gradient Ricci solitons with positive curvature operator.\\[-1ex]

\item \label{example-2} {\bf Ricci flows coming out of non-collapsed Ricci limit spaces.}\\[0.5ex]
Let $(M_i,g_i(0),x_i)_{i\in \N}$ be a sequence of smooth $n$-dimensional Riemannian manifolds with bounded curvature,  such that 
 $\curlR(g_i(0)) + c\cdot \Id(g_i(0)) \in \Cone_{\curlK}$ and $\vol(B_{g_i(0)}(x)) \geq v_0$ for all $x \in M_i$ 
for all $i \in \N$, for some $c,v_0>0$  where $\curlR $ is the curvature operator, $\Id$ is the identity operator of the sphere and $\Cone_{\curlK}$ is the cone of  i) non-negative curvature operators, respectively ii) $2$-non-negative curvature operators,
 respectively iii) weakly $PIC_1$  curvature operators,  respectively iv) weakly $PIC_2$ curvature operators.
Then [\cite{MilesCrelle3D} for (i), (ii) in case $n=3$,  \cite{BamCab-RivWil} for (i) -(iv) for general $n\in \N$] 
there are  solutions  $(M_i,g_i(t),x_i)_{t\in [0,T(n,v_0,c)]}$ such that
$\curlR(g_i(t))) + C\cdot\Id \in \Cone_{\curlK}$ (for some new $C>0$). Note that this implies $\Rc(g_i(t)) \geq -c(n)C.$ After scaling each solution we obtain a sequence of solutions satisfying  \eqref{eq:intro.1} and \eqref{basic-ass-curv-bd-intro}. 
Taking a  sub-sequencial limit, we obtain a pointed Cheeger-Hamilton limit solution $(M^n,g(t),x_{\infty})_{t \in (0,1)}$ which satisfies \eqref{eq:intro.1} and \eqref{basic-ass-curv-bd-intro}.\\[1ex]
More generally, if we take a sequence of smooth complete solutions $(M_i,g_i(t),x_i)_{t\in [0,1)}$ satisfying \eqref{eq:intro.1} and \eqref{basic-ass-curv-bd-intro}, and 
$\vol(B_{g_i(t)}(x_i),1) \geq v_0$ for all $i\in \N$ for all $t\in [0,1)$, we obtain a pointed solution $(M^n,g(t),x_{\infty})_{t \in (0,1)}$ as a sub-sequential Cheeger-Hamilton limit, which satisfies  \eqref{eq:intro.1} and \eqref{basic-ass-curv-bd-intro}.\\
\end{enumerate}

\noindent {\bf Local setting.} Problem \ref{main-pb} can be considered  locally in the context of the above examples as follows.\\[-2ex]
\begin{enumerate}
\item Assume that  $(M,g(t) )_{t\in (0,\infty)}  $ is a smooth self-similarly expanding solution with non-negative Ricci curvature coming out of a cone $(M^n,d_X)  = (\R^+ \times S^{n-1} ,dr^2 \oplus r^2 \ga) $, where $\ga$ is a smooth (continuous)  Riemannian metric. Does the solution $(M,g(t))_{t\in (0,1)}$ come out smoothly (continuously)? 
That is, is the solution
$$ (M \backslash \{p\}, g(t))_{t\in [0,1]}$$ smooth (continuous), where $p$ is the tip of the cone, and $g_0$ is the cone metric on $M \backslash \{p\}$ at time zero?\\
If we replace the assumption that '$\ga$ is smooth (continuous) on $S^{n-1}$' to '$\ga$ is smooth  (continuous)  on an open set $V \subseteq S^n$', we ask the question: is $$ ((\R^+ \times V ), g(t)|_{\R^+ \times V} )_{t\in [0,1]}$$ smooth (continuous)?

\item In the setting of Example \eqref{example-2} let $(M,d_0,x_{\infty})$ be the limit as $t \downto 0$ of 
 $(M,d(g(t)), x_{\infty}) $. Note that $(M,d_0, x_{\infty})$ is isometric to the 
Gromov-Hausdorff limit of $(M_i,d(g_i(0)),x_i)$ as $i \to \infty$, in view of \eqref{distestinit}.
Let $V \subseteq M$ be an open set such that $(V,d_0)$ is isometric to a smooth (continuous) Riemannian manifold. Can $(V,g(t))_{t\in (0,1)}$ be extended smoothly (continuously) to $t=0$, that is, does there exist a smooth (continuous) $g_0$ on $V$ such that $(V,g(t))_{t\in [0,1)}$ is smooth (continuous)?\\[-1ex]
\end{enumerate}

\noindent We will see in Theorem \ref{smoothness_of_solutions_intro}, that the answer to each of these questions in the smooth setting  is {\bf yes}, if we measure the smoothness of the initial metric space appropriately. 
The answer to each of these questions is also {\bf yes} in the continuous setting, see Theorem \ref{continuity_of_solutions_intro}, if we measure the continuity of the initial metric  appropriately {\it and} the convergence in the continuous setting is measured  {\it up to diffeomorphisms}.

The smoothness (respectively continuity) of a metric space in this paper will be measured as follows.
Let $0<\ep_0(n)<100^{-1}$ be a small fixed positive constant depending only on $n$. We denote with $\B_{r}(x) \subset \R^n$ the Euclidean ball with radius $r$, centred at $x$. 

\begin{defn}\label{smoothness_continuity_metric_spaces}
Let $(X,d_0)$ be a  metric space and let $V$ be a set in $X$. We say $(V,d_0)$  is {\it smoothly {\rm(respectively} continuously}) $n$-Riemannian 
 if  for all $x_0 \in V$ there exist $0<\ti r,r$ with $ \ti r < \frac 1 5 r$ and points $a_1, \ldots, a_n \in B_{d_0}(x_0,r)$ such that the map
 $$F_0(x):= (d_0(x,a_1), \ldots, d_0(x,a_n)),\quad x\in B_{d_0}(x_0,r), $$ is a $(1+\ep_0)$ bi-Lipschitz homeomorphism on $B_{d_0}(x_0,5 \ti r)$ and the push-forward of $d_0$ via $F_0$ given by $\ti d_0(\ti x, \ti y):= d_0((F_0)^{-1}(\ti x), (F_0)^{-1}(\ti y))$ on $\B_{4 \ti r}(F_0(x_0))$ $\Subset F_0(B_{d_0}(x_0,5 \ti r))$  is induced by a smooth (respectively {\it continuous}) Riemannian metric: there exists a {\it smooth} (respectively {\it continuous}) Riemannian metric $\ti g_0$ defined on $\B_{4 \ti r}(F_0(x_0)),$ such that $\ti d_0$ satisfies $\ti d_0 = d (\ti g_0 )$, when restricted to $\B_{\ti r}(F_0(x_0))$, where $d(\ti g_0)$ is the distance on  $( \B_{4 \ti r}(F_0(x_0)), \ti  g_0)$.
\end{defn}

Since this definition might be slightly counter-intuitive at first reading, we give an alternative definition as well.

\begin{defn}\label{smoothness_continuity_metric_spaces_second}
Let $(X,d_0)$ be a  metric space and let $V$ be an open set in $X$. We say $(V,d_0)$  is {\it smoothly} $n$-Riemannian 
 if  for all $x_0 \in V$ there exist   a smooth $n$-dimensional  connected Riemannian manifold
 $(U(x_0),\ti g_0)$ and a neighbourhood $V(x_0) \subseteq V $ of $x_0$, and an isometry 
  $F_0:(V(x_0),d_0) \to (F_0(V(x_0)),\ti d_0) \subseteq (U(x_0),\ti d_0),$ where $(U(x_0),\ti d_0)  = 
  (U(x_0), d(\ti g_0)),$
  and $d(\ti g_0)$ is the distance on  $(U(x_0), \ti  g_0)$.
\end{defn}
\noindent {\bf Equivalence of definitions.} The first definition clearly implies the second one. That the second definition implies the first, may be seen as follows:
Since $(U(x_0),\ti g_0)$ is smooth, we may find $ \ti a_1, \ldots \ti a_n \in F_0(V(x_0))$ such that
$\ti F_0(\ti x) = ( \ti d_0(\ti a_1, \ti x), \ldots, \ti d_0(\ti a_n, \ti x))$ is $(1+\ep_0)$ bi-Lipschitz and smooth in a neighbourhood of $F(x_0)$. 
Defining $ a_1:= F_0^{-1}(\ti a_1), \ldots, a_n:= F_0^{-1}(a_n)),$  and $\hat g_0:= (\ti F_0)_*(\ti g_0) ,$
we see that $\hat F_0 = \ti F_0 \circ F_0 $ satisfies: $(\hat F_0)_*(d_0) := \hat d_0 = d(\hat g_0)$ and
$$\hat F_0(x) = (  \ti d_0(\ti a_1,  F_0(x)), \ldots, \ti d_0(\ti a_n, F_0( x) ) )
= (   d_0( a_1, x), \ldots, d_0(a_n,x) )$$ is $(1+\ep_0)$ bi-Lipschitz, as required.

\subsection{Main results}
The first theorem gives a positive answer to the questions posed in the previous subsection in the  smooth setting. 

\begin{thm}\label{smoothness_of_solutions_intro}
Let $(M,g(t))_{t\in (0,T]}$ be a smooth solution to Ricci flow satisfying \eqref{eq:intro.1} and \eqref{basic-ass-curv-bd-intro}, 
and assume $B_{g(t)}(x_0,1) \Subset M$ for all $0<t\leq T$ and 
let $(X,d_0)$ be the $C^0$ limit,as $t\downto 0$ of $(X,d(g(t))$ established in Lemma  \ref{SiTopThm1}, where $X:= (\cap_{s \in (0,T)} B_{g(s)}(x_0,\frac 1 2)).$  Assume further that $(B_{d_0}(x_0,r),d_0)$ is smoothly $n$-Riemannian  in the sense of Definition
  \ref{smoothness_continuity_metric_spaces}, where $r \leq R(c_0,n)$ and $R(c_0,n)$ is as in Lemma \ref{SiTopThm1}. Then there exists a smooth Riemannian metric $g_0$ on $B_{d_0}(x_0,s)$ for some $s\in(0,r)$ such that we can extend  the smooth solution $(B_{d_0}(x_0,s),g(t))_{t\in (0,T)}$ to 
a smooth solution $(B_{d_0}(x_0,s),g(t))_{t\in [0,T)}$  by defining $g(0) = g_0$.
\end{thm}

As noted  by Topping in \cite{Top-Pri-Com}, this result was  known for  Ricci flow of closed 2-manifolds by results of Richard in \cite{Richard-T-Alex}. 

The second theorem is concerned with the corresponding question in the continuous setting.

\begin{thm}\label{continuity_of_solutions_intro}
Let $(M,g(t))_{t\in (0,T]}$ be a smooth solution to Ricci flow satisfying \eqref{eq:intro.1} and \eqref{basic-ass-curv-bd-intro}, 
and assume $B_{g(t)}(x_0,1) \Subset M$ for all $0<t\leq T$ and 
let $(X,d_0)$ be the $C^0$ limit,as $t\downto 0$ of $(X,d(g(t))$ established in Lemma  \ref{SiTopThm1}, where $X:= (\cap_{s \in (0,T)} B_{g(s)}(x_0,\frac 1 2)).$  Assume further that $(B_{d_0}(x_0,r),d_0)$ is continuously $n$-Riemannian  in the sense of Definition
  \ref{smoothness_continuity_metric_spaces}, where $r \leq R(c_0,n)$ and $R(c_0,n)$ is as in Lemma \ref{SiTopThm1}.

Then for any strictly monotone sequence  $t_i \searrow 0$, there exists a radius $v>0$ and a continuous Riemannian metric $\ti g_0$, defined on $\B_{v}(p),$ $p\in \R^n$, and a   family of smooth diffeomorphisms $Z_{i}:B_{d_0}(x_0,2v) \to \R^n$ such that $(Z_{i})_{*}(g(t_i))$ converges in the $C^0$ sense
to $\ti g_0$ as $t_i \downto 0$ on $\B_{v }(p)$.

\end{thm}

\subsection{Metric space convergence and the conditions \eqref{eq:intro.1} and \eqref{basic-ass-curv-bd-intro}}

Assume we have a smooth complete solution to Ricci flow 
$(M^n,g(t))_{t \in (0,1)},$ satisfying \eqref{eq:intro.1} but not necessarily \eqref{basic-ass-curv-bd-intro}. Then there is no guarantee that a limit metric $d_0= \lim_{t \to 0} d(g(t))$ exists.
Similarly, if we have  a sequence of smooth complete solutions $(M_i^n,g_i(t),x_i)_{t \in [0,1)},$  satisfying  $|\Rm (g_i(t)))|\leq c_0/t$ and $\vol(B_{g_i(t)}(x)) \geq v_0 >0$ for all $t\in (0,1)$ for all $x\in M_i$ for some $c_0,v_0 >0,$  for all $i\in \N$, we obtain 
a limiting solution in the smooth Cheeger-Hamilton sense, $(M^n,g(t),p)_{t \in (0,1)},$ which satisfies   
$|\Rm (\cdot,t)|\leq c_0/t$ for all $t\in (0,1)$, but again, there is {\it no} guarantee that a limit metric $d_0= \lim_{t \to 0} d(g(t))$ exists.
Furthermore, \textit{if} a pointed Gromov-Hausdorff limit $(M,d_0,p)$, as $t\to 0$, of $(M,d(g(t)),p)$ exists and \textit{if} a Gromov-Hausdorff limit
$(X,d_X,y) $ in $i\in\N$ of $(M_i^n,d(g_i(0)),x_i)_{i\in \N}$ exists, then  there is no guarantee that  $(X,d_X,y) $ is isometric to $(M,d_0,p),$ or that $(M,d_0)$ has the same topology as $d(g(t))$ for $t>0$.

 An example which considers the metric behaviour under limits of solutions with no uniform bound from below on the Ricci curvature but with $|\Rm(\cdot,t)| \leq c_0/t$ is given in a  recent work of Peter Topping \cite{Top-Pri-Com}. 
There, he  
constructs   examples of smooth solutions $(\T^2,g_i(t))_{t\in [0,1)}$ to Ricci flow, satisfying  
$(\T^2,d(g_i(0)))  \to (T^2,d(\de))$,  as $i\to \infty,$ where $\de$ is the standard flat metric on $\T^2,$ and  $|\Rm(g_i(t))|\leq c/t$  for all $t\in (0,1), i\in \N$ for some $c>0$, 
but so that the limiting solution $(\T^2,g(t))_{t\in (0,1]} $ satisfies $(\T^2,d(g(t)) )_{t\in (0,1)}  = (\T^2,\hat d),$   
where $(\T^2,\hat d)$ is isometric to 
 $(T^2,d(2\de)).$
The initial smooth data $g_i(0)$ do not satisfy $\Rc(g_i(0)) \geq -k$ for some fixed $k>0$ for all $i\in \N,$ and so the arguments used to show  that the Gromov-Hausdorff limit of the initial data  is the same as the limit as $t \to 0$ of the limiting solution, are not valid.

\subsection{Related results}
We recall the basic setup for the initial trace problem for the scalar heat equation.
Consider a smooth solution $u: \R^n \times (0,1) \to \R$, with  $\partt u = \lap u$,
and assume that $u(\cdot,t) \to u_0(\cdot)$ in $C^0_\text{loc}(\B_1(0))$ as $t\searrow 0$, where $u_0 \in C^\infty(\B_1(0))$. Then the solution can be locally extended to a smooth local solution $v: \B_{1/2}(0)  \times [0,1) \to \R$, by defining $v(\cdot,0) = u_0(\cdot)$ on $\B_{1/2}(0)$, as we now explain:
since $u \in C^0(\overline{\B_{3/4}(0)}  \times [0,1])$, by standard theory, there exists  a solution of the heat equation $z \in C^0(\overline{\B_{3/4}(0)}  \times [0,1]) \cap C^\infty(\B_{3/4}(0) \times [0,1])$, such that
  $z = u$ on the parabolic boundary $\boundary  \B_{3/4}(0) \times [0,1]\,  \cup \, 
 \overline{\B_{3/4}(0)} \times \{0 \} $.  The maximum principle then implies that $z \equiv u$ and hence $u$ is smooth on $\B_{3/4}(0)  \times [0,1]$, as required.
Here the linear theory simplifies the situation. We have also assumed that  $u(\cdot,t) \to u_0$ locally uniformly.
In the Ricci flow setting, assuming \eqref{eq:intro.1} and \eqref{basic-ass-curv-bd-intro}, we saw above  that the initial values must be taken on uniformly, albeit for the distance, not necessarily the Riemannian metric.

A non-linear  setting closer to the one considered in this paper is as follows.
In \cite{Appleton}, Appleton considers (among other things)  the $\de$-Ricci-DeTurck flow of metrics $g_0$ on $\R^n$ which are close to the standard metric $\de$, in the sense that $|g_0 -\de|_{\de} \leq \ep(n)$. 
In the work of Koch and Lamm, see \cite[Theorem 4.3]{KochLamm}, it was shown that under this closeness condition there always exists a {\it weak solution}  $(\R^n,g(t))_{t\in (0,\infty)}$.   Weak solutions defined on $[0,T)$ ($T = \infty$ is allowed) 
are smooth for all $t>0$ and $h(x,t):= g(x,t) -\de(x)$ 
has    bounded $X_T$ norm, where
\begin{equation*}
\begin{split}
\| h& \|_{X_T} :=  \sup_{0<t<T} \|h(t)\|_{\de} \cr
& \ \ \ \ \ + \sup_{ \! \! x \in \R^n} \sup_{0< R^2 <T}  \Big(R^{-\frac{n}{2}} \|\grad h\|_{L^2(\B_R(x) \times (0,R^2))}  + R^{\frac{2}{n+4} } \|\grad h\|_{L^{n+4}( \B_R(x) \times (\frac{R^2}{2},R^2)) } \Big)
\end{split}
\end{equation*}
If the initial values $g_0$  are continuous then the  initial values are attained in the $C^0$-sense, that is $|g(t)-g_0|_{\de} \to 0$ as $t \to 0$. 
Appleton showed, see \cite[Theorem 4.5]{Appleton}, that  any weak solution $h(t):= g(t)-\de$ which has $g_0 \in C^{2,\al}_{loc}(\R^n)$ and $|h_0|_{\de} \leq \ep(n),$  
must have $h(t) \in H^{2 +\al, 1 + \frac{\al}{2} }_{loc}(\R^n \times [0,\infty))$. In particular the zeroth, first and second spatial derivatives  of $h(t)$ locally approach  
those of $h_0$ as $t\downto 0$.  
That is, for classical initial data $h_0 \in  C^{2,\al}_{loc}(\R^n),$ any weak solution $h(t)$ restricted to $\Omega$ approaches $h(0)$  in the $C^{2,\al}(\Omega)$  norm on $\Omega$ for any precompact, open set $\Omega$.

Rapha\"{e}l Hochard established in \cite{HochardThesis} some results similar to some of those appearing in Sections \ref{harm} and \ref{bi_lip_section} of the current paper.
We received a copy of Hochard's thesis, after a pre-print version, including  the  relevant sections,  of this paper was finished but not yet published.
We have included references to the results of Hochard at the appropriate points throughout this paper. His approach differs slightly, as we explain at the relevant points.

\subsection{Outline of paper}
We outline the idea of proof of the main theorems, Theorem  \ref{smoothness_of_solutions_intro} and \ref{continuity_of_solutions_intro}.
The idea is somewhat similar to the one we used above to show smoothness of solutions to the heat equation coming out of smooth initial data, which are smooth for positive times.

It is well known that since Ricci flow is invariant under diffeomorphisms, it only represents a degenerate parabolic system. To able to prove initial regularity, it is thus necessary to put Ricci flow into a good gauge, transforming Ricci flow into a strictly parabolic system. The strategy we follow here is to construct a local family of diffeomorphims solving Ricci-harmonic map heat flow into $\R^n$ and push forward the Ricci flow solution to obtain a solution to $\de$-Ricci-DeTurck flow. 

More precisely, despite the low initial regularity, we show that there is a solution to the Ricci-harmonic map heat flow, $Z \in C^0 (\overline{B_{d_0}(x_0,1)}  \times [0,T); \R^n) \cap  C^\infty (B_{d_0}(x_0,1)  \times (0,T); \R^n) $, with initial and boundary values given by the map $F_0$, which represents distance coordinates at time zero. The a priori estimates we prove in Sections \ref{harm}    and  \ref{bi_lip_section} help us to construct this solution, and from the Regularity Theorem \ref{flowthm}, we see that there is $S(n)>0$ and a small $\al(n)>0$ such that the solution is  $1+\al(n)$ bi-Lipschitz for each $t \in (0,S(n)) \cap [0,T/2)$. The explicit construction of $Z$ is carried out in Theorem \ref{RicciDeTurck}.

Hence, we may consider the push forward $\ti g(t) : = (Z_t)_*(g(t)),$ which is by construction a solution to $\de$-Ricci-DeTurck, and $\ti g(t)$ is $\al(n)$ close to $\de$ in the $C^0$-sense. We first restrict to the case that  the push forward of $d_0$ with respect to $F_0$ is generated locally by a continuous
Riemannian metric $\ti g_0$. 
A further application of the regularity theorem, Theorem \ref{flowthm}, yields that $\ti g(t)$ converges locally to the continuous metric $\ti g_0$. This is explained in detail in Theorem \ref{continuous_thm} in Section \ref{continuous_solutions}.

If we assume further that $\ti g_0$ is smooth, and sufficiently close to $\de$, 
then we consider the Dirichlet Solution  $\ell$ to the $\de$-Ricci-DeTurck flow on an Euclidean ball $\B_r(0) \times [0,T] $, with parabolic boundary data given by $\ti g$. The existence of this solution is shown in  Section \ref{Dirichlet_Section}, where Dirichlet solutions to the $\de$-Ricci-DeTurck flow with given parabolic boundary values $C^0$ close to $\de$ are constructed. 
The $L^2$-lemma, Lemma \ref{L2Lemma} of Section \ref{L2Lemma_Section}, tells us that  the (weighted)  spatial $L^2$ norm of the difference $g_1-g_2$ of two solutions $g_1,g_2$ to the $\de$-Ricci-DeTurck flow defined on an Euclidean ball is  non-increasing, if $g_1$ and $g_2$ 
have the same values on the boundary of that ball, and are sufficiently close to $\de$ for all time $t \in [0,T]$.
An application of the $L^2$-lemma  then proves that $\ell = \ti g$. 
The construction of $\ell$, carried out in Section \ref{Dirichlet_Section} guarantees that $\ell$ is smooth on $\B_r(0) \times [0,T]$. Hence $\ti g$ is smooth on $\B_r(0) \times [0,T]$. Section \ref{application} completes the proof of Theorem \ref{smoothness_of_solutions_intro}: the smoothness of $\ti g$  on $\B_r(0) \times [0,T]$ implies that one can extend $g$ smoothly (locally) to $t=0$. In Section \ref{application} we discuss some of the consequences of Theorem \ref{smoothness_of_solutions_intro} in the context of  expanding gradient Ricci solitons with non-negative Ricci curvature.

\subsection{An open problem} The lower bound on the Ricci curvature in \eqref{basic-ass-curv-bd-intro} is used crucially 
to obtain the bound from above for  $d_t$ in \eqref{distestinit}. It is also used  in Section \ref{continuous_solutions}, when showing  that $\ti g(t)$ converges to $\ti g_0$ in the $C^0$ norm. 
\begin{pb}
Can the bound from below on the Ricci curvature in Section \ref{bi_lip_section} and/or other sections be  replaced by a weaker condition?
\end{pb}
\noindent We comment  on this at various points in the paper.

\subsection{Notation} We collect notation used throughout this paper.
\begin{itemize}
\item[(1)] For a connected Riemannian manifold $(M,g),$ $x,y \in M,$ $r\in \R^+$:\\[-2ex]
\begin{itemize}
\item[(1a)]  $(M,d(g))$ refers to the associated metric space, 
$$d(g)(x,y) = \inf_{\ga \in G_{x,y}}L_g(\ga),$$ where $G_{x,y}$ refers to the set of smooth regular curves $\ga:[0,1] \to M,$ with
$\ga(0) = x,$ $\ga(1) = y$, and $L_g(\ga)$ is the length of $\ga$ with respect to $g$.
\item[(1b)] $B_{g}(x,r) := B_{d(g)}(x,r) := \{ y \in M \ | \ d(g)(y,x) < r\}$.
\item[(1c)] If $g$ is locally in $C^2$:  $\Rc(g)$ is the Ricci Tensor, $\Rm(g)$ is the Riemannian curvature tensor,  and  $\Sc(g)$ is the scalar curvature.\\[-2ex]
\end{itemize}
\item[(2)] For a one parameter family $(g(t))_{t\in(0,T)}$ of Riemannian metrics on a manifold $M$, the distance induced by the metric $g(t)$ is denoted either by $d(g(t))$ or $d_t$ for $t\in(0,T)$.\\[-2ex]
\item[(3)] For a metric space $(X,d),$ $x \in M$, $r\in \R^+,$  $B_{d}(x,r) := \{ y \in M \ | \ d(y,x) < r\}.$\\[-2ex]
\item[(4)] $\B_{r}(x)$ refers to an Euclidean ball with radius $r>0$ and centre $x \in \R^n$.
\end{itemize}

\subsection{Acknowledgements}
A.~Deruelle is supported by grant ANR-17-CE40-0034 of the French National Research Agency ANR (Project CCEM) and Fondation Louis D., Project "Jeunes G\'eom\`etres". F.~Schulze is supported by a
Leverhulme Trust Research Project Grant RPG-2016-174. M.~Simon is supported by the SPP 2026 'Geometry at Infinity' of the German Research Foundation (DFG). 

The authors are grateful to the anonymous referees for carefully reading the previous version of this paper, and for their suggestions and comments. 
 These suggestions and comments led to changes which we believe  have  improved  the exposition of the paper.

\section{Ricci-harmonic map heat flow for functions with bounded gradient}{\label{harm}}

\noindent In this section we prove some local results about the Ricci-harmonic map heat flow. 

R.~Hochard, in independent work, proved some results in his PhD-thesis which are similar to some results in this chapter, see \cite[Section II.3.2]{HochardThesis}. Hochard  uses blow up arguments to prove some of his estimates, whereas we use a more direct argument involving the maximum principle applied to various evolving quantities.

The first theorem we present  is a local version of a theorem of Hamilton, \cite[p.~15]{HamFor}, for solutions satisfying \eqref{eq:intro.1} and \eqref{basic-ass-curv-bd-intro}.
\begin{thm} \label{thm-grad-est}
Let $(M^n,g(t))_{t\in [0,T]}$ be a smooth background solution to Ricci flow satisfying \eqref{eq:intro.1} and \eqref{basic-ass-curv-bd-intro} such that
  $B_{g(0)}(x_0,2) \Subset  M$ and $\boundary B_{g(0)}(x_0,1)$  is a smooth $(n-1)$-dimensional manifold. \\[1ex]
Let $Z_0: \overline{B_{g(0)}(x_0,1)}  \to \R^n$  be a smooth map such that 
\begin{itemize}
\item $|\grad Z_0|_{g(0)}^2 \leq c_1$  
\item $Z_0(\overline{B_{g(0)}(x_0,1)}) \subseteq \B_r(0)$ for some $r\leq 2$. 
\end{itemize}
 Then there is a unique  solution $$Z\in C^{\infty}(B_{g(0)}(x_0,1) \times [0,T]; \R^n) \cap 
C^{0}(\overline{ B_{g(0)}(x_0,1) } \times [0,T]; \R^n),$$ to the Dirichlet problem for the Ricci-harmonic map heat flow 
 \begin{equation}
 \begin{split} 
 & \tfrac{\partial}{\partial t} Z  = \lap_{g(t),\de} Z,   \\
 & Z(\cdot,0) = Z_0, \\
 & Z(\cdot,t)|_{\boundary B_{g(0)}(x_0,1)}  = Z_0|_{\boundary B_{g(0)}(x_0,1)}  \ \ \mbox{ for }  t\in [0,T],\label{Z-Dir-Pb}
 \end{split}
 \end{equation}
and constants $c(c_0,c_1,n), S(n,c_0)>0$ such that
\begin{gather} 
 Z_t(\overline{B_{g(0)}(x_0,1)}) \subseteq \B_{r}(0),\quad\mbox{  for all $ t\leq T$,}\label{Z-subset}\\ 
B_{g(t)}\left(x_0, 3/4\right) \subseteq B_{g(0)}(x_0, 1),\quad \text{for all $t\leq \min(T,S(n,c_0))$,}\label{comparison-geo-ball-diff-time}\\
 |\grad^{g(t)} Z(\cdot,t)|_{g(t)}^2 \leq c(c_0,c_1,n),\label{grad-est} 
\end{gather}
on $B_{g(t)}\left(x_0, 1/2\right)$ for all $t\leq \min(T,S(n,c_0))$. 
\end{thm}
\begin{proof} We first note that the system for $Z$ actually decouples into $n$ independent linear equations. Since $\overline{B_{g(0)}(x_0,1)}  \subset  M$ is a compact set, and the solution $(M,g(t))_{t\in [0,T]}$ is smooth, by standard theory there is a  unique solution $$Z\in C^{\infty}(B_{g(0)}(x_0,1) \times [0,T]; \R^n) \cap 
C^{0}(\overline{ B_{g(0)}(x_0,1) } \times [0,T]; \R^n),$$ to the Dirichlet problem \eqref{Z-Dir-Pb}.

For the sake of clarity, we omit the dependence of the Levi-Civita connections on the metrics $(g(t))_{t\in[0,T]}$.

The statement \eqref{Z-subset} follows from the Maximum Principle and the evolution equation for $|Z|^2 = \sum_{i=1}^n (Z^i)^2$: 
\begin{equation}
\left(\partt -\lap_{g(t)}\right)|Z|^2  =  -2|\grad Z|_{g(t)}^2. \label{evo-equ-Z}
\end{equation}
\\
Statement \eqref{comparison-geo-ball-diff-time} follows from the distance estimates, 
\eqref{distestinit}, which  hold on  $B_{g(0)}(x_0,1)$ for any solution to Ricci flow satisfying \eqref{eq:intro.1}, \eqref{basic-ass-curv-bd-intro} and $B_{g(0)}(x_0,2) \Subset M$: see  \cite[Lemma 3.1]{SiTo1}.

Regarding \eqref{grad-est}, we first recall the following fundamental evolution equation satisfied by $|\nabla Z|^2_{g(t)}$:
\begin{equation}
\left(\partt -\lap_{g(t)}\right)|\nabla Z|_{g(t)}^2  =  -2|\grad^2 Z|_{g(t)}^2. \label{grad-evo-equ}
\end{equation}
Notice that the term $\Rc(g(t))(\nabla Z,\nabla Z)$ showing up in the Bochner formula applied to $\nabla Z$ cancels with the pointwise evolution equation of the squared norm of $\nabla Z$ along the Ricci flow. 

In case the underlying manifold is closed, the use of the maximum principle would give us the expected result.

In order to localize this argument, we construct a Perelman type cut-off function $\eta: M \to [0,1]$ with $\eta(\cdot,t) = 0$ on $M\backslash B_{g(t)}(x_0,\frac 3 4)$ and
$ \eta(\cdot,t) = e^{-k(n,c_0) t}$ on $B_{g(t)}(x_0,\frac 1 2)$  such that $\partt \eta(\cdot,t) \leq \lap_{g(t)} \eta (\cdot,t)$ everywhere, and $|\grad \eta |_{g(t)}^2 \leq c_3(n) \eta $ everywhere, as long as $t\leq \min(S(n,c_0),T)$ : see, for example, \cite[Section 7]{SiTo1} for details. 

We consider the function $W:= \eta  |\grad Z|_{g}^2 + c_2|Z|^2 $, with 
$c_2 = 10 c(n) c_3(n).$  The quantity $W$ is less than $c_1 + 4c_2$ everywhere at time zero.
We consider a first time and point where  $W$  becomes equal to $c_1 + 5c_2$ on $\overline{ B_{g(0)}(x_0,1)}$. 
This must happen  in $B_{g(0)}(x_0,1)$, since $\eta =0$  on a small open set $U$  containing 
$\boundary B_{g(0)}(x_0,1) $ and $c_2|Z|^2 <  4c_2$ by \eqref{Z-subset}.  At such a point and time $(x,t)$, we have by \eqref{evo-equ-Z} and \eqref{grad-evo-equ} together with the properties of $\eta$,
\begin{equation*}
\begin{split}
0 &\leq  \left(\partt -\lap_{g(t)}\right) W (x,t) \\
& \leq  - 2c_2 |\grad Z|_{g(t)}^2 
-2\eta |\grad^2 Z|_{g(t)}^2   - 2 g(t)( \grad \eta, \grad |\grad Z|_{g(t)}^2) \\
& \leq - 2c_2 |\grad Z|_{g(t)}^2 
-2\eta |\grad^2 Z|_{g(t)}^2   + 4c(n)\frac{|\grad \eta|_{g(t)}^2}{\eta}  |\grad Z|^2_{g(t)} + \eta |\grad^2 Z|_{g(t)}^2\\
& < 0,
\end{split}
\end{equation*}
 by the choice of $c_2$, which yields a contradiction. Hence $W(x,t) \leq c_1 + 5c_2$ for all $t \leq S(n,c_0),$ which implies
\begin{equation*}
  e^{-kt}|\grad Z|_{g(t)}^2(\cdot,t) \leq e^{-kt}|\grad Z|_{g(t)}^2(\cdot,t) + c_2|Z|^2(\cdot,t) \leq (c_1 + 5c_2),
  \end{equation*}
   on $B_{g(t)}(x_0,\frac 1 2)$ for all $t \leq \min(S(n,c_0),T)$. This gives
$$ |\grad Z|_{g(t)}^2(\cdot,t) \leq e^{kS}(c_1 + 5c_2),$$ on $B_{g(t)}(x_0,\frac 1 2)$ for all $t \leq \min(S(n,c_0),T)$ as required.
\end{proof}

We aim to prove an estimate for the second covariant derivatives of a solution to the Ricci-harmonic map flow. In fact, once we have a solution to the Ricci-harmonic map heat flow with bounded gradient, the solution smoothes out the second derivatives in a controlled way, as the following theorem shows.

\begin{thm}\label{reg_harm}
For all $c_1>0$ and $n\in \N$, there exists $\hat \varepsilon_0(c_1,n)>0$ such that the following is true.
Let $(M^n,g(t))_{t\in [0,T]}$ be a smooth solution to Ricci flow such that
\begin{eqnarray*}
\Rc(g(t))  \geq -1,\quad |\Rm(\cdot,t)| \leq \frac{\varepsilon^2_0}{t},\quad \text{ for all } t\in(0,T],
\end{eqnarray*}
  where $\ep_0 \leq \hat \ep_0$.
Assume furthermore that  $B_{g(0)}(x_0,1) \Subset M$,  
and $ Z: B_{g(0)}(x_0,1) \times [0,T] \to \R^n$ is a smooth solution to the Ricci-harmonic map heat flow 
$$ \partt Z(x,t) = \lap_{g(t)} Z(x,t),$$ for all $(x,t) \in B_{g(0)}(x_0,1) \times [0,T]$,
such that 
$|\grad^{g(t)} Z(\cdot,t)|_{g(t)}^2 \leq c_1$ on $B_{g(t)}(x_0, 1)$ for all $t\in [0,T]$.
Then 
 \begin{eqnarray*}
  t|\grad^{g(t),2} Z(\cdot,t)|_{g(t)}^2 \leq c(n,c_1),
 \end{eqnarray*}
 on $B_{g(t)}(x_0, 1/4)$ for all $t\leq \min\{S(n),T\}$, where $S(n)>0$ is a constant just depending on $n$.
\end{thm}
\begin{rmk}
The condition $|\Rm(\cdot,t)| \leq \varepsilon^2_0/t $ where $\varepsilon_0 \leq \hat \varepsilon_0(c_1,n)$ is sufficiently small
is not necessary : $|\Rm(\cdot,t)| \leq k/t $ with  $k$ arbitrary is sufficient for the argument, as can be seen by examining the proof, but then the conclusions remain only valid on the time interval
$(0,S(c_1,k,n)]$, where $S(c_1,k,n)>0$ is sufficiently small. We only consider small $k$, as this is sufficient for the setting of  the following chapters. A version of this theorem, with $c_1 =c(n)$ and the condition $|\Rm(\cdot,t)| \leq k/t, $ $k$ arbitrary, was independently proven by R.~Hochard using a contradiction argument: see \cite[Lemma II.3.9]{HochardThesis}.
\end{rmk}



\begin{proof}
In the following, we denote constants $C(\varepsilon_0,c_1,n)$ simply by $\xi_0$ if 
$C(\varepsilon_0,c_1,n) $ goes to $0$ as $\varepsilon_0$ tends to $0$  and $c_1$ and $n$ remain  fixed. 
For example $c_1^2 n^4 \varepsilon_0$ and  $b(c_1,n) \sqrt{\varepsilon_0} $ are denoted by $\xi_0$ if $b(c_1,n)$ is a constant depending on $c_1$ and $n$, and $d(c_1,n)\sqrt{\xi_0} $ may be replaced by $\xi_0,$
if $d(c_1,n)$ is a constant depending on $c_1$ and $n$ only.\\[1ex]
For the sake of clarity in the computation to follow, we use the notation $\grad$ to denote $\grad^{g(t)}$ at a time $t$, $\Riem$ to denote $\Riem(g(t))$ at a time $t$, et cetera, although the objects in question do indeed depend  on the evolving metric. By standard commutator identities for the second covariant derivatives of a tensor $T$,
$\grad^2 T (V,P, \cdot) = \grad^2 T (P,V, \cdot) + (\Rm *T )(V,P,\cdot)$: see for example \cite{Jeff_Lecture_Notes}, we obtain
\begin{equation*}
\begin{split}
\partt \grad^2_{ij} Z^k & = \grad^2_{ij} \Big(\partt Z\Big)^k + (\grad \Rc * \grad Z )^k_{ij}\\
& =  \grad^2( \lap Z)_{ij}^k + (\grad \Rc * \grad Z )^k_{ij}\\
& = \lap ( \grad_i \grad_j Z^k ) +  (\Rm * \grad^2 Z)_{ij}^k +  (\grad \Riem * \grad Z)^k_{ij}.
\end{split}
\end{equation*}
Shi's estimates and the distance estimates \eqref{distestinit} guarantee that $|\grad \Rm(g(t))| \leq \xi_0t^{-3/2}$ for $t \leq S(n)$
on $B_{g(t)}(x_0,3/4)$ : see for example   Lemma \ref{RicciFlowBasics} (after scaling once by $400$).
On $B_{g(t)}(x_0,3/4)$, we see for $t \leq S(n)$, that
\begin{equation}
\begin{split} 
\partt |\grad^2 Z|_{g(t)}^2 &  \leq  \lap_{g(t)} ( |\grad^2 Z|_{g(t)}^2)    - 2|\grad^3 Z|_{g(t)}^2 + c(n) |\Rm(g(t))|_{g(t)}|\grad^2 Z|_{g(t)}^2  \\
&\quad + c(n) |\grad \Rm(g(t))|_{g(t)}|\grad^2  Z|_{g(t)} |\grad Z|_{g(t)}\\
& \leq  \lap_{g(t)} ( |\grad^2 Z|_{g(t)}^2)    - 2|\grad^3 Z|_{g(t)}^2  + \frac{\xi_0}{t} |\grad^2 Z|_{g(t)}^2 \\
&\ \ \  +\frac{\xi_0}{t^{\frac 3 2}} | \grad^2  Z|_{g(t)} |\grad Z|_{g(t)} \\
 & \leq  \lap_{g(t)} ( |\grad^2 Z|^2_{g(t)})    - 2|\grad^3 Z|_{g(t)}^2 +  \frac{\xi_0}{t} |\grad^2 Z|_{g(t)}^2 + \frac{\xi_0}{t^2} |\grad Z|_{g(t)}^2.\label{evo-equ-hessian-Z}
 \end{split}
 \end{equation}
 For $a_0 \geq 1$, let $$W:= t(a_0 + |\grad Z|_{g(t)}^2 ) |\grad ^2 Z|_{g(t)}^2.$$ 
 Using \eqref{grad-evo-equ} together with \eqref{evo-equ-hessian-Z},  we see that 
 \begin{equation*}
 \begin{split} 
 \left(\partt -\lap_{g(t)}\right) W  &= \left(\partt -\lap_{g(t)}\right)\Big(  t(a_0 + |\grad Z|_{g(t)}^2 ) |\grad^2 Z|_{g(t)}^2 \Big) \\
 &= (a_0 +  |\grad Z|_{g(t)}^2) |\grad^2 Z|_{g(t)}^2 +t\left(\partial_t-\Delta_{g(t)}\right)\left(|\nabla Z|^2_{g(t)}\right)\cdot|\nabla^2Z|^2_{g(t)}\\
 &\quad +t(a_0+|\nabla Z|^2_{g(t)})\left(\partial_t-\Delta_{g(t)}\right)|\nabla^2 Z|^2_{g(t)}-2 t   g(t)\left( \grad |\grad Z|_{g(t)}^2, \grad |\grad^2 Z|_{g(t)}^2\right)\\
&  \leq (a_0 +  |\grad Z|_{g(t)}^2) |\grad^2 Z|_{g(t)}^2   - 2 t|\grad^2 Z|_{g(t)}^4- 2a_0t|\grad^3 Z|_{g(t)}^2 \\
&\quad +(1+a_0)\xi_0 |\grad^2 Z|_{g(t)}^2  + \frac{(1+a_0)\xi_0}{t}  -2 t   g(t)\left( \grad |\grad Z|_{g(t)}^2, \grad |\grad^2 Z|_{g(t)}^2\right) \\
& \leq   \frac{W}{t}  - 2 t|\grad^2 Z|_{g(t)}^4 - 2a_0t|\grad^3 Z|_{g(t)}^2+ a_0\xi_0 |\grad^2 Z|_{g(t)}^2 \\
&\quad + \frac{a_0\xi_0}{t} + t  c(n,c_1) |\grad^3 Z|_{g(t)} |\grad^2 Z|_{g(t)}^2 \\
& \leq  \frac{W}{t} + (c(n,c_1)-2a_0)t|\grad^3 Z|_{g(t)}^2  -  t|\grad^2 Z|_{g(t)}^4 + a_0\xi_0 |\grad^2 Z|_{g(t)}^2
+ \frac{a_0\xi_0}{t}\\
& \leq \frac{W}{t} + (c(n,c_1)-2a_0)t|\grad^3 Z|_{g(t)}^2  -  \frac{t}{2}|\grad^2 Z|_{g(t)}^4 
+ \frac{a^2_0\xi_0}{t},\\
\end{split}
\end{equation*}
where $c(n,c_1)$ denotes a positive constant depending on the dimension $n$ and the Lipschitz constant $c_1$, that may vary from line to line and we have used Young's inequality freely. Thus
\begin{equation}
\begin{split}
\partial_tW&\leq \lap_{g(t)} W+  \frac{W}{t} - \frac{W^2}{2(a_0+ c_1)^2t}  + \frac{a^2_0 \xi_0}{t},\cr
& \leq   \lap_{g(t)} W+  \frac{W}{t} - \frac{W^2}{4a_0^2t}  + \frac{a^2_0 \xi_0}{t}, \label{sub-sol-W}
\end{split}
\end{equation}
if $a_0$ is chosen sufficiently large such that $a_0\geq c(n,c_1)$.

In case the underlying manifold is closed, the use of the maximum principle would give us the expected result:  if there is a first time and point $(x,t)$ where $W(x,t)  = 10a_0^2$ for example, we obtain a contradiction. Hence we must have
$W \leq 10a_0^2$.

In order to localize this argument, we introduce, as in the proof of Theorem \ref{thm-grad-est}, a Perelman type cut-off function $\eta: M \to [0,1]$ with $\eta(\cdot,t) = 0$ on $B^c_{g(t)}(x_0,2/3)$ and
$ \eta(\cdot,t) = e^{-t}$ on $B_{g(t)}(x_0,1/2)$  such that $\partt \eta(\cdot,t) \leq \lap_{g(t)} \eta (\cdot,t)$ everywhere, and $|\grad \eta |_{g(t)}^2 \leq c(n) \eta $ everywhere, as long as $t\leq S(n) \leq 1$: see for example \cite[Section 7]{SiTo1} for details.

We first derive the evolution equation of the function $\hat{W}:=\eta W$ with the help of inequality \eqref{sub-sol-W} for $t\leq S(n)$:
\begin{equation}
\partial_t\hat{W}\leq \Delta_{g(t)}\hat{W}-2\langle\nabla \eta,\nabla W\rangle_{g(t)} + \frac{\hat W}{t} - \eta\frac{W^2}{4 a^2_0t}+ \eta\frac{a_0^2\xi_0}{t} \, .
\label{evo-equ-W-cut-off}
\end{equation}
Thus
\begin{equation*}
\begin{split}
\eta\left(\partial_t-\Delta_{g(t)}\right)\hat{W}&\leq   -2\eta\langle\nabla \eta,\nabla W\rangle_{g(t)} +  \frac{\eta \hat W}{t} -\frac{\hat{W}^2}{4a_0^2t}+ \eta^2\frac{a^2_0 \xi_0}{t}\\
&\leq -2\langle\nabla \eta,\nabla\hat{W}\rangle_{g(t)}+c(n)\hat{W} +  \frac{\hat W}{t}- \frac{\hat{W}^2}{4 a^2_0t}+ \frac{a^2_0\xi_0}{t}\ \label{final-evo-equ-loc-W}
\end{split}
\end{equation*}
where again, $c(n)$ denotes a positive constant depending on the dimension only and which may vary from line to line. If the maximum of $\hat{W}$ at any time is larger that $100a_0^2$ then this value must be achieved  at some first time and point $(x,t)$ with $t>0$
, since $\hat W(\cdot,0) = 0$. This leads to a contradiction if $t \leq S(n) \leq 1$.
 \end{proof}

 \section{Almost isometries, distance coordinates  and Ricci-harmonic map heat flow  }\label{bi_lip_section}
 
In this first subsection we set up the problem we will be investigating for the remaining chapter, collect some background material and give an outline of the further strategy.

For convenience we recall a slightly refined version of Lemma \ref{SiTopThm1} in the introduction, where $\beta(n)>0$ is the constant appearing in \cite[Lemma 3.1]{SiTo2}.

\begin{lemma}(Simon-Topping, \cite[Lemma 3.1]{SiTo2}).
Let $(M,g(t))_{t\in(0,T)}$, $T \leq 1$, be a smooth Ricci flow where 
 $M$ is connected but $(M,g(t))$ not  necessarily complete. Assume that for some $\ep_0>0$ and  $R>100\be^2(n)\ep^2_0+200$ it holds that $B_{g(t)}(x_0,200R) \Subset M $ for all $t\in (0,T)$ as well as\\[-3ex]
\begin{align} 
& |\Riem(\cdot,t)| \leq \frac{\ep^2_0}{(1+\be^2(n))t}\ \  \mbox{on} \ \  B_{g(t)}(x_0,200R) \Subset M  
 \ \mbox{ for all } \  t \in (0,T)  \tag{$\tilde{a}$}  \label{tildea}\\
& \Rcci(g(t))\geq -1 \ \  \mbox{on} \ \   B_{g(t)}(x_0,200R) \Subset M  \ \mbox{ for all } \  t \in (0,T).  \tag{$\tilde{b}$} \label{tildeb}
 \end{align}
Then for all $r,s \in (0,T)$ it holds that
\begin{equation*}
 e^{t-r}d_r  \geq d_t  \geq d_r - \ep_0 \sqrt{t-r} \qquad \forall\,t\in [r,T) \mbox{ on }  B_{g(s)}(x_0,50R) .
\end{equation*}

\end{lemma}
\medskip
This motivates for us to work with the setup that $(M,g(t))_{t\in (0,T)}$ is a smooth solution to Ricci flow, with 
\begin{align} 
& |\Riem(\cdot,t)| \leq \ep^2_0/t\ \  \mbox{on} \ \  B_{g(t)}(x_0,200R) \Subset M  
 \ \mbox{ for all } \  t \in (0,T)  \tag{${a}$}  \label{a}\\
& \Rcci(g(t))\geq -1 \ \  \mbox{on} \ \   B_{g(t)}(x_0,200R) \Subset M  \ \mbox{ for all } \  t \in (0,T)  \tag{${b}$}  \label{b}\\
& e^{t-r}d_r  \geq d_t  \geq d_r - \ep_0 \sqrt{t-r} \qquad \forall\,t\in [r,T) \mbox{ on }  B_{g(s)}(x_0,50R)    \ \ \forall\   r,s \in (0,T) \label{distest}  \\
&  B_{g(s)}(x_0,50R)   \Subset M \mbox{ for all } s \in (0,T). \label{distincl} 
  \end{align}
As in Lemma \ref{SiTopThm1}, if \eqref{distest} and \eqref{distincl}  hold, then  there exists a unique limiting  metric as $t\downto 0$:
\begin{equation}
\begin{split}
& \ \ d_0 = \lim_{t\downto 0} d_t  \mbox{ exists, is unique and is a metric  on}  \   X:=  \cap_{s \in (0,T)} B_{g(s)}(x_0,50R),  \label{definitiond_0}  \\[-1.5ex]
&  \ \    B_{d_0}(x_0,20R) \subseteq \curlX \subseteq   B_{g(s)}(x_0,50R)   \Subset M, \forall\,  s \in  \left(0, \min\left( \tfrac{R^2}{1+\ep_0^2 },  \log(2),T\right ) \right ),  \cr
&  \  \ \mbox{and the topology of }  B_{d_0}(x_0,20R) \mbox{ induced by }  d_0  \mbox{ agrees with  the topology ≠} \\ 
& \ \ \mbox{of the set }  B_{d_0}(x_0,20R) \subseteq M \mbox{ induced by } M,
\end{split}
\end{equation}
where $\curlX$ is the connected component of the interior of $\cap_{s\in (0,T)} B_{g(s)}(x_0,50R) $ containing $x_0,$ and the limit is obtained uniformly: the existence of a unique limit follows from the fact that $d_t(\cdot,\cdot)$ is Cauchy in $t$, and the inclusions and statement about the topology 
follow from the triangle inequality and the definition of the interior of a set.

The estimates of Shi yield the following time interior decay.
\begin{lemma}\label{RicciFlowBasics}
Let $(M,g(t))_{t\in(0,T)}$ be a smooth, not necessarily complete,  solution to Ricci flow  satisfying \eqref{a}, \eqref{b},  \eqref{distest} and \eqref{distincl}  for some $R \geq 1.$ Then
 \begin{equation}
\sum_{i=0}^j |\grad^i  \Riem(\cdot,t)|^2 \leq \frac{\be(k,n,\ep_0)}{t^{2+j}} \label{shi_est}
\end{equation}
for all $x \in B_{d_0}(x_0,10R)$ for all $t \in (0, \min(  (1+\ep_0^2)^{-1}R^2,  \log(2),T) ),$
where $\be(k,n,\ep_0) \to 0$ for fixed $k$ and $n$, as $\ep_0 \to 0,$ and $d_0$ is the metric described in \eqref{definitiond_0}.
\end{lemma}
\begin{proof}
We scale the solution $\ti g(\cdot,\ti t):= t^{-1}g(\cdot,\ti t t)$, so that time $t$ in the original solution scales to time $\ti t$ equal to $1$ in the new one. We now have   $|\widetilde{\Rm}(\cdot,s)|\leq 2\ep^2_0$ on $B_{\ti g(1/2)}(x,1)$ for all $s \in [1/2, 2],$ for all $x \in B_{d_0}(x_0,10R),$  in view of \eqref{a} and the (scaled) distance estimates \eqref{distest}.
The estimates of Shi, see for example \cite[Theorem 6.5]{CLN},
 give us that $\sum_{i=0}^k |\grad^i  \Riem(x,1)|^2 \leq \be(k,n,\ep_0)$  where $\be(k,n,\ep_0) \to 0$ as $\ep_0 \to 0$, as claimed. 
\end{proof}
\medskip
In the main theorem of this chapter, Theorem \ref{flowthm}, we consider distance maps $F_t$ at time $t$,  points $m_0\in B_{d_0}(x_0,15R)$ and $\ep_0< \frac 1 2$ such that $F_t : B_{d_0}(m_0,10) \to \R^n$  satisfies 
\begin{align}
  &|F_t(x)-F_t(y)| \in \Big( (1-\ep_0)d_{t}(x,y) -\ep_0 \sqrt{t}, (1+\ep_0)d_{t}(x,y) +\ep_0 \sqrt{t} \Big)\tag{c} \label{c}
  \end{align}
  for all $x,y \in  B_{d_0}(m_0,10)$.
  If such a map $F_t$ exists for all $t\in (0,T)$, and we further assume that
  $\sup_{t\in (0,T)} |F_t(x_0)| < \infty$, then for any sequence $t_i>0$, $t_i \to 0$ with $i\to \infty,$ we can, after taking a subsequence, find a limiting map, $F_0$ which is the $C^0$ limit of $F_{t_i},$
   $\sup_{x\in B_{d_0}(m_0,10)}  |F_0(x)-F_{t_i}(x) |\to 0,$ as $i\to \infty$  which satisfies 
  \begin{equation}
(1-\varepsilon_0)  \leq \frac{|F_{0}(x) -F_{0}(y)|}{d_0(x,y)} \leq (1+\varepsilon_0) \label{bilipproperty}
\end{equation}
on $ B_{d_0}(x_0,10)$.
Indeed, we first define $F_0$ on a dense, countable subset $ D \subset B_{d_0}(x_0,10)$ using a diagonal subsequence and the theorem of Heine-Borel, and then we extend $F_0$ uniquely, continuously to all of $B_{d_0} (x_0,10)$, which is possible in view of the fact that  the Bi-Lipschitz property \eqref{bilipproperty} is satisfied on $D$. The sequence $(F_{t_i})_{i}$ converges uniformly to $F_0$ in view of \eqref{distest}, \eqref{c} and \eqref{bilipproperty}.

Thus $F_0 $ is a $(1+\varepsilon_0)$ Bi-Lipschitz map between the metric spaces 
 $ (B_{d_0}(x_0,10),d_0)$ and $( F_0(B_{d_0}(x_0,10) ),\de)$. This is equivalent  to 
  $H_0 (\cdot) = F_0(\cdot) -F_0(x_0) $ being a $1+\varepsilon_0$ Bi-Lipschitz map between the metric spaces 
 $ (B_{d_0}(x_0,10),d_0)$  and $H_0(B_{d_0}(x_0,10))$ where 
 $   \B_{5}(0)  \subseteq  H_0(B_{d_0}(x_0,10)) \subseteq \B_{20}(0).$

 In  Theorem \ref{flowthm}, we see that if we consider a   Ricci-harmonic map heat flow of one of the functions $F_t$, and we assume that the solution $Z$ satisfies a gradient bound, $|\grad^{g(s)} Z(\cdot,s)|_{g(s)} \leq c_1$, for $s \in [t,T]$ on some ball, then after flowing for a time $t$, the resulting map will be a $1+ \al_0$   Bi-Lipschitz map on a smaller ball,  if $\ep_0 $ is less than a constant $ \hat \ep_0(n,\al_0,c_1)>0$, and $t\leq  \min(S(n,\al_0,c_1),T)$.  This property continues to hold if we flow for a time $s$ where $t \leq s \leq  \min(S(n,\al_0,c_1),T)$.

For convenience, we introduce the following notation:

\begin{defn} 
Let $(W,d)$ be a metric space. We call  $F:(W,d) \to \R^n$ an $\ep_0$ almost isometry if
\begin{equation*} 
(1-\varepsilon_0)d(x,y)  -\ep_0  \leq  |F(x) -F(y)| \leq (1+\varepsilon_0)d(x,y) +\ep_0
\end{equation*}
for all $x,y \in W$.
$F:(W,d) \to \R^n$ is a  $1+\ep_0$ Bi-Lipschitz map, if
 \begin{equation*} 
(1-\varepsilon_0)d(x,y)  \leq  |F(x) -F(y)| \leq (1+\varepsilon_0)d(x,y) 
\end{equation*}
for all $x,y \in W$.
\end{defn}
In the main applications in this chapter  of Theorem \ref{flowthm} and Theorem \ref{RicciDeTurck}, we will assume:
\begin{align*}
& \mbox{ there are  points } a_1, \ldots, a_n \in B_{d_0}(x_0,R), \mbox{  such that 
the map } \cr
& F_0: \, \left\{
   \begin{array}{rcl}
       B_{d_0}(x_0,R)  & \rightarrow &\R^n \\
       x& \rightarrow & (d_0(a_1,x), \ldots, d_0(a_n,x)):= ((F_0)_1(x), \ldots, (F_0)_n(x))
   \end{array}\right. \tag{$\hat{\rm c}$} \label{chat} \\
&  \mbox{ is a }   (1+\ep_0)  \mbox{ Bi-Lipschitz homeomorphism on }  B_{d_0}(x_0,100).  \\[-2ex]
\end{align*}
The components of the map $F_0$  are   referred to as {\it distance coordinates}. 
As a consequence of this assumption and the distance estimates \eqref{distest}, we see that  the corresponding distance coordinates at time $t$, $F_t:   B_{d_0}(x_0,50) \to \R^n $, given by $F_t(\cdot):= (d_t(a_1,\cdot), \ldots, d_t(a_n,\cdot))$,  are mappings  satisfying property \eqref{c}, for all $t\in (0,T)$.
That is: in the main application, we begin with a $1+\ep_0$ Bi-Lipschitz map $F_0$ and find, as a first step, maps $F_t$ which are  $\ep_0$ almost isometries for all $t\in (0,T)$.
\smallskip

\begin{rmk}
R.~Hochard also looked independently at some related objects in his PhD-thesis, and some of the infinitesimal results he obtained there are similar to those of this section, cf.~\cite[Theorem II.3.10]{HochardThesis}, as we explained in the introduction. Hochard considers points  $x_0$ which are so called {\it $(m,\ep)$ explosions at all scales less then $R$}  (only $m=n$ is relevant in this discussion). The condition, for $m=n$,  says that there exist points $p_1, \ldots, p_n$ such that for all $x$ in the ball $B_{d_0}(x_0,R)$ and all $r<R$,
 there exists an $\ep r$  GH approximation $\psi: B_{d_0}(x,r) \to \R^n$  such that the components $\psi^i$  are  each close to the components of distance coordinates $d(\cdot,p_i) - d(x,p_i)$ at the scale $r$, in the sense that $|\psi^i(\cdot)- (d(\cdot,p_i) - d(x,p_i) )| \leq \ep r$ on $B_{d_0}(x,r)$. 
  Our approach and our main conclusion differ slightly to the approach and main conclusions  of Hochard. The condition  \eqref{c}  we consider above looks 
 at the closeness of the maps  $F_t$ to being a Bi-Lipschitz homeomorphism, and  
 this closeness is measured at time $t$ using the maps $F_t$, and our main conclusion, is that the map will be a $1 + \al_0$ Bi-Lipschitz homeomorphism after flowing for an appropriate time by Ricci-harmonic map heat flow, if $t>0$ is small enough.
  We make the assumption on the evolving curvature, that  it is  close to that  of $\R^n$, after scaling in time appropriately. 
  Nevertheless, the proof of Theorem \ref{flowthm} below and of  \cite[Theorem II.3.10]{HochardThesis}   have a number of similarities, as do some of the concepts.
\end{rmk}

\noindent {\bf Outline of the section.} The main application of this section is, assuming \eqref{a}, \eqref{b} and that $F_0$ are distance coordinates which define a $1+\ep_0$ Bi-Lipschitz homeomorphism, to show that it is possible to define a Ricci-DeTurck flow $(\ti g(s))_{s\in (0,T]}$ starting from the metric $\ti d_0  := (F_0)_* d_0$,  on some Euclidean ball, which is obtained by pushing forward the solution $(g(s))_{s\in (0,T]}$ by diffeomorphisms. The sense in which this is  to be understood, respectively this is true, will be explained in more detail in the next paragraph.
  
The strategy we adopt is as follows. Assuming \eqref{a}, \eqref{b}, we consider the distance maps  $F_{t_i}$ at time $t_i$ defined above for a sequence of times $t_i>0$ with $t_i \to 0.$
We  mollify each $F_{t_i}$ at an appropriately  small  scale, so that they become  smooth, but so that the essential property, \eqref{c},  of the $F_{t_i}$ is not lost (at least up to a factor $2$). 
Then we flow each of the $F_{t_i}$ on $B_{d_0}(x_0,100)$ by Ricci-harmonic map heat flow, keeping the boundary values fixed. The existence of the solutions $Z_{t_i}:B_{d_0}(x_0,100) \times [t_i,T) \to \R^n$, is   guaranteed by Theorem  \ref{thm-grad-est}.
According to Theorem 
\ref{flowthm}, the $Z_{t_i}(s)$  are then $1+\al_0$ Bi-Lipschitz maps  (on a smaller ball), for all $s\in [2t_i,S(n,\ep_0)] \cap (0,T/2),$
if $\ep_0 = \ep_0(\al_0,n)$ is small enough. 
If we take the push forward of $g(s)$ with respect to $Z_{t_i}(s),$  $s\in [2t_i,S(n,\ep_0)] \cap (0,T/2),$ (on a smaller ball), then after taking a limit of a subsequence in $i$, we obtain a solution $\ti g(s)_{s \in (0,S(n,\ep_0)) \cap (0,T/2)}$ to the $\de$-Ricci-DeTurck flow such that 
$(1-\al_0)\de \leq \ti g(s) \leq (1+\al_0)\de $ for all $s \in(0,S(n,\ep_0)) \cap (0,T/2)$  in view of the  estimates  of Theorem \ref{flowthm}.
The solution then satisfies $d(\ti g(t)) \to \ti d_0 := (F_0)_* (d_0)$  as $t\to 0$ and hence may be thought of as a solution to Ricci-DeTurck flow coming out of $\ti d_0$.
This is explained in Theorem \ref{RicciDeTurck}.

In the next  section we examine the  regularity properties of this solution, which depend on the regularity properties of  $\ti d_0$.

\subsection{Almost isometries and Ricci-harmonic map heat flow} In this subsection we provide some technical lemmas giving insight into the evolution of almost isometries under Ricci-harmonic map heat flow. These results will be needed in the following subsection.

\begin{lemma}\label{almostisomlem}
For all $\si$, there exists an $0<\ga(\si)\leq \si$ small with the following property: 
if $ L: \B_{{\ga}^{-1}}(0 ) \to \R^n$ is a  $\ga$ almost isometry fixing $0$,  then  
there exists an  $S \in O(n)$ such that
$|L-S|_{L^{\infty}(\B_{\si^{-1} }(0)) } \leq \si$. 
\end{lemma}

\begin{proof}
If not, then for some $\si>0$, we have a sequence of maps 
$L_i: \B_{i^{-1}}(0 ) \to \R^n,$ $i\in \N$ such that $L_i$ is an almost  $i^{-1}$ isometry fixing $0$, but
$L_i$ is not $\si $ close in  the $L^{\infty}$ sense to any element $S \in O(n)$ on 
$\B_{\si^{-1} }(0)$. 
Let $N:=2\si^{-1}$ and $D \subseteq \B_{N}(0)$ be a dense subset of $\B_{N}(0)$.
By taking a diagonal subsequence and  using that $|L_i|_{L^{\infty}(\B_{N}(0))}\leq N+1$ in conjunction with the Theorem of Bolzano-Weierstra{\ss}, we obtain a map $L:D \to \R^n$ satisfying $L(x):= \lim_{i\to \infty} L_i(x)$ for all $x \in D$. $L:D\to \R^n$ satisfies $|L(x)-L(y)| = |x-y|$ for all $x,y \in D,$ and hence may
be continuously extended to a map $L:B_N(0) \to \R^n$ which is an isometry 
$|L(x)-L(y)| = |x-y| $ for all $x,y \in B_N(0)$. Using the facts that the $L_i's$ are almost isometries,  $L_i \to L$ poinwise on $D$, and $L$ is an isometry,  we see that in fact $L_i \to L$ uniformly on $B_{N}(0)$ for $i\to \infty,$  
which contradicts the fact that $L_i$  is not $\si$ close to any element $S \in O(n)$ on $B_{\si^{-1}}(0)$.
\end{proof}


 \begin{lemma}\label{BasicLemma}
 For all $c_1, \al_0>0$, $n \in \N$  there exists $0 < \al(c_1,n,\al_0) \leq \al_0 $ such that the following is true.
 Let $Z:\B_{\al^{-1}}(0) \times [0,1] \to \R^n$ be a smooth solution to the harmonic map heat flow with an evolving background metric, 
\begin{eqnarray*}
\partt Z(\cdot,t) = \lap_{h(t)} Z(\cdot,t) ,
\end{eqnarray*}
 where $h(\cdot,t)_{t\in [0,1]}$ is smooth, and $Z_0(0) =0$, where $Z_0(\cdot):= Z(\cdot,0)$. We  assume  that $Z_0$ is an  $\al$ almost isometry with respect to $h(0)$, and further  that:
 \begin{gather*}
  |h_{ij} -\de_{ij}|_{C^2(\B_{\al^{-1}}(0) \times [0,1]) }  \leq \al \\
 |\grad^{h(t)} Z(\cdot,t)|_{h(t)} \leq c_1 \ \mbox{ on } \B_{\alpha^{-1}}(0) \\
  |\grad^{h(t),2} Z(\cdot,t)|_{h(t)} \leq \frac{c_1}{\sqrt{t}} \  \mbox{ on } \B_{\alpha^{-1}}(0)
\end{gather*}
for all $t\in [0,1]$. Then 
\begin{equation}\label{BasicLemmaIneq}
\begin{split}
|dZ(\cdot,s)(v)| &\in (1-\al_0, 1+\al_0)|v|_{h(s)} \\
|Z(x,s) -Z(y,s) | &\in \big( (1-\al_0)d(h(s))(x,y), (1+\al_0)d(h(s))(x,y) \big)\\
|Z_0(x) -Z(x,t)| &\leq \al_0\, ,
\end{split}
\end{equation}
 for all $s \in [1/2,1]$, for all $t\in [0,1],$ for all $x,y \in \B_{\alpha_0^{-1}  }(0)$ and $v \in T_x \B_{\al_0^{-1}}(0)$.
\end{lemma}

\begin{proof}
We omit the dependence of the Levi-Civita connections on the metrics $h(t)$, $t\in[0,1]$.
From Lemma \ref{almostisomlem}, we know that there exists an $S \in O(n)$ such that 
\begin{equation*} 
|Z(\cdot,0) -S(\cdot)|_{C^0(\B_{\be^{-1}}(0) )}  \leq \be
\end{equation*}
where $0<\al \leq \be = \be(n,c_1,\al)$, but still $\be(n,c_1,\al) \to 0$ as $\al \to 0$.

We also know that  $|\partt Z(x,t)| = |\lap_{h(t)} Z(x,t)|  \leq c_1/\sqrt{t}$ and 
hence
$|Z(x,t) -Z(x,0)|\leq 2c_1 $ 
 for all $t \in [0,1]$  for all $x \in \B_{\be^{-1}}(0) $ which implies
$$|Z(x,t) -S(x)| \leq |Z(x,t) - Z(x,0)| + |Z(x,0) -S(x)| \leq 3 c_1$$
 (w.l.o.g.~$\be \leq c_1$) for all $t \in [0,1]$  for all $x \in \B_{\be^{-1}}(0) $.

Let $\eta:\B_{\be^{-1}}(0) \to [0,1] $ be a smooth cut off function such that $\eta(\cdot) = 1$ on $\B_{\be^{-1}/2}(0)$, $\eta(\cdot) = 0$ on $\R^n \backslash \B_{\beta^{-1}}(0),$  
$|D^2\eta| + (|D \eta |^2/\eta )\leq c(n)\be^2$, $|D\eta|\leq c(n)\be$.
Due to the fact that $h$ is $\al$ close to $\de$ in the $C^2$ norm, we see 
$$|\lap_{h}(S)(x,t)| = |h^{ij}(x,t)( \partial_i \partial_j S(x) - \Gamma(h)_{ij}^k(x,t) \partial_k S )| \leq c(n) \al\, .$$
Hence, we have 
\begin{equation*}
 \partt (Z-S)(\cdot,t)  = \lap_{h(t)} (Z-S)(\cdot,t) + E,
 \end{equation*}
 where $|E|  \leq  c(n) \al$.
But then, using the  cut off function $\eta$ on $\B_{\be^{-1}}(0)$, we get
\begin{equation*}
\begin{split}
\partt (|Z-S|^2 \eta)
 &\leq  \lap_{h(t)}(|Z-S)|^2 \eta)  -2\eta  | \grad (Z -S)|^2  - |Z-S|^2 \lap_{h(t)} \eta\\
& \quad + c(n)\alpha  |Z- S| -2h(t)(\grad |Z-S|^2, \grad \eta ) \\
&\leq   \ \lap_{h(t)}(|Z-S)|^2 \eta)+ c(n) \beta |Z-S|^2 - \eta | \grad (Z -S)|^2  
+c(n)c_1 \al \\ 
&\leq  \ \lap_{h(t)}(|Z-S|^2 \eta) + c(n, c_1) \be\, .
\label{maxprinciple}
\end{split}
\end{equation*}
Hence, by the maximum principle, $|Z-S|^2(t) \leq \al^2 + c(n, c_1)\be \leq c(n, c_1) \be $ for all $t\in[0,1]$ on
$\B_{\be^{-1}/2}(0)$.

Since $|\grad (Z -S)|^2 + |\grad^2 ( Z -S )|^2 \leq c(n) c_1$  for $s \in [1/2,1]$  on
$\B_{\be^{-1}/2}(0)$ we can use interpolation inequalities as in \cite[Lemma B.1]{ChodoshSchulze} to deduce that on $\B_{\be^{-1}/4}(0)$
$$ |D (Z-S)|_\de \leq c(n, c_1) \be^\frac{1}{4}\, .$$ 
Again, since $h(t)$ is $\al$-close to $\de$ this implies that for $\al$ sufficiently small 
\begin{equation*}
|d Z(\cdot,s) (v) |\in (1- \al_0, 1+ \al_0)
\end{equation*}
for all $v \in T_y\R^n$ of length one with respect to $h(s)$  for $s \in [1/2,1]$ and $y \in \B_{\be^{-1}/4}(0)$.
We also see that
$$|Z(x,s) - Z(y,s)| = |DZ(p,s)(v_p)||x-y| \in ( d(h(s))(x,y) (1-\al_0), d(h(s))(x,y) (1+\al_0) )$$
where by the mean value theorem  $p $ is some  point on the unit speed line between $x$ and $y$, and $v_p$ is a vector of length one with respect to $\de$ and 
$$ |Z(x,t)-Z_0(x)| \leq |Z(x,t) -S| + |Z_0(x)-S| \leq c(n, c_1) \beta^\frac{1}{2} \leq \alpha_0 $$
for all $x,y \in \B_{\be^{-1}/4}(0)$, for all $s\in [1/2,1]$, $t\in [0,1]$ for $\al$ sufficiently small.
\end{proof}

\smallskip

\begin{lemma}\label{DGLemma}
For all  $n,k \in \N$, $L >0$ there exists an $\ep_0 = \ep_0(n,k,L)>0$ such that the following holds.
Let $M^n$ be a connected smooth manifold, and $g$ and $h$ be smooth Riemannian metrics on $M$  with
 $B_{h}(y_0,L) \Subset M$ and $|d_h-d_g|_{C^0(B_{h}(y_0,L))} \leq \ep_0$. 
Assume further that 
$$\sup_{B_{h}(y_0,L)} \big( |\Riem(g)|_{g} + \ldots + | \grad^{g,k+2} \Riem(g)|_{g} \big)\leq \ep_0$$ and that there exists a map  $F: B_{h}(y_0,L) \to \R^n$ which is an $\ep_0$ almost isometry with respect to $h$,
 that is 
 \begin{equation*} 
(1-\varepsilon_0)d_h(z,y)  -\ep_0  \leq  |F(z) -F(y)| \leq (1+\varepsilon_0)d_h(z,y) +\ep_0,
\end{equation*}
 for all $z,y \in  B_{h}(y_0,L),$  and $F(y_0) = 0$. Then  $ (B_{h}(y_0,L/2),g)$ is $1/L$-close to
 the Euclidean ball $(\B_{L/2}(0),\de)$ in the $C^k$-Cheeger-Gromov sense.
 \end{lemma}

 \begin{proof}
Assume it is not the case. Then there is an $L>0$ for which the theorem fails.  Then we can find sequences $g(i),h(i)$,$M(i), F(i): B_{h(i)}(y_i,L) \to \R^n $ satisfying the above conditions with $\ep_0:= 1/ i$ but so that the conclusion of the theorem is not correct. Using the almost isometry, we see that for any $\ep>0$  we can cover $ B_{h(i)}(y_0(i),6L/7)$ by $N(\ep)$ balls (with respect to $h$) of radius $\ep$, for all $i$. Hence, using 
 $|d_{h(i)}-d_{g(i)}|_{C^0(B_{h}(y_0,L))} \leq 1/i$, we see that the same is true for $ B_{g(i)}(y_0(i),5L/6) \subseteq B_{h(i)}(y_0(i), 6L/7)$ with respect to $g(i)$: we can cover $ B_{g(i)}(y_0(i),5L/6 )$ by $N(\ep)$ balls (with respect to $g(i)$) of radius $\ep$, for all $i$.
Hence, due to the compactness theorem of Gromov (see for example \cite[Theorem 8.1.10]{BBi}),  there is a Gromov-Hausdorff  Limit 
$$(X,d) = \lim_{i \to \infty} (B_{g(i)}(y_0(i),4L/5),g(i)) = \lim_{i \to \infty} (B_{h(i)}(y_0(i),4L/5),h(i))\, .$$
In particular, there must exist 
$$G(i): B_{d}\Big(y_0,\frac{3L}{4}\Big) \to B_{h(i)}\Big(y_0(i),\frac{4L}{5}\Big)\, ,$$ 
which are $\ep(i)$ Gromov-Hausdorff approximations, where $\ep(i) \to 0$   as $i\to \infty$.
Using the maps $G(i)$ and the $1/i$ almost isometries $F(i)$, we see that there is a pointwise limit map, $H:= \lim_{i\to \infty} F(i) \circ G(i), $
$$H:\Big(B_{d}\Big(y_0,\frac{3L}{4}\Big), d\Big) \to \big(\B_{4L/5}(0) ,\de\big)\, ,$$ which is an isometry.  Hence the volume of
$ B_{g(i)}(y_0(i),2L/3)$ converges to $\omega_n (2L/3)^n$ (in particular the sequence  is non-collapsing) as $i\to \infty$, since volume is convergent for spaces of bounded curvature (which are for example Aleksandrov spaces and spaces with Ricci curvature bounded from below).
Hence $(B_{g(i)}(y_0(i), L/2 ),g(i))$ converges to   $(\B_{ L/2 }(0),\de) $ in the $C^k$ norm in the Cheeger-Gromov sense, which leads to a  contradiction if $i$ is large enough. \end{proof}

\subsection{Ricci-harmonic map heat flow of $(1+\ep_0)$-Bi-Lipschitz maps and  distance coordinates}
We begin with a regularity theorem for solutions to the Ricci-harmonic map heat flow, whose initial values are sufficiently close to a $1+\ep_0$ Bi-Lipschitz map.

\begin{thm}\label{flowthm}
For all  $\al_0 \in (0,1), n\in \N$ and $c_1 \in \R^+$, there exists $ \ep_0(n,c_1,\al_0)>0$ and $ S(n,c_1,\al_0)>0$ such that the following holds. 
Let $(M,g(t))_{t\in (0,T)}$ be a smooth solution to Ricci flow satisfying the conditions \eqref{a}, \eqref{distest} and \eqref{distincl} for some $R>100$, where  $d_0$ is the metric appearing in \eqref{definitiond_0}.
If $m_0 \in  B_{d_0}(x_0,15R)$ and $Z:B_{d_0}(m_0,10) \times [t,T) \to \R^n$ is a solution to the Ricci-harmonic map heat flow 
$$\parts Z = \lap_{g(s)} Z$$
for some $t\in [0,T)$ on $B_{d_0}(m_0,10)$ for all $s \in [t,T)$,  which satisfies $|\grad^{g(s)} Z|_{g(s)} \leq c_1$ on $B_{d_0}(m_0,10)$ for all $s \in[t,T)$, 
 and the initial values of $Z$ satisfy \eqref{c}, that is 
\begin{equation}
|Z(x,t) -Z(y,t)| \in ( (1-\ep_0)d_t(x,y) -\ep_0 \sqrt{t}, (1+\ep_0)d_t(x,y) +\ep_0 \sqrt{t}) \label{initprop}
\end{equation}
for all $x,y \in B_{d_0}(m_0,10),$ then
\begin{gather} 
 | \grad ^{g(s),2} Z(x,s)|^2_{g(s)} \leq \al_0 (s-t)^{-1}  \label{flow_second_gradient}  \\
|Z(x,s)-Z(x,t)| \leq c(c_1,n)\sqrt{s-t} \ \label{flow_c0time} \\
|dZ(x,s)(v) | \in 
(1-\al_0,1+ \al_0)  |v|_{g(s)} ,\label{flow_gradest}  \\
 |Z(x,s) -Z(y,s)| \in ((1-\al_0)d_{s}(x,y), (1+ \al_0)d_{s}(x,y))  \label{flow_bilip}  
\end{gather}
for all $x,y \in B_{d_0}(m_0,2)$, $v \in T_x B_{d_0}(m_0,2)$, $s \in [2t,S(n,c_1,\al_0)] \cap [0,T/2] $. 
Furthermore, the maps  $Z(s): B_{d_0}(m_0,3/2) \to D_s:= Z(s)(B_{d_0}(m_0,3/2)) \subseteq \R^n$ for $s \in [2t,S(n,c_1,\al_0)]\, \cap \,[0,T/2]$ are homeomorphisms and their image satisfies 
$$\B_{5/4}(Z(s)(m_0))  \subseteq D_s \subseteq \B_{2}(Z(s)(m_0)),$$ for $Z(s):= Z_s(\cdot)$.
\end{thm}

A direct consequence of \eqref{flow_gradest} is the following corollary.

\begin{cor} \label{corflowthm}Assuming the set up of Theorem  \ref{flowthm}, the following metric inequalities hold
\begin{gather}
  (1-\al_0)^2  g(x,s) \leq   (Z(s))^*\delta  \leq  (1+\al_0)^2 g(x,s) \label{firstin} \\
 (1-\al_0)^2  (Z(s))_*g(y,s) \leq \de  \leq (1+\al_0)^2 (Z(s))_*g(y,s) \label{secondin}
   \end{gather}
 for all $x \in B_{d_0}(m_0,2)$, for all 
 $y \in  \B_{5/4 }(Z(s)(m_0))  \subseteq Z(s)( B_{d_0}(m_0,2)) \subseteq \R^n $ and for all $s\in [2t, S(n,c_1,\al_0)]  \cap [0,T/2] $. 
 \end{cor}

\begin{proof}[Proof of Theorem \ref{flowthm}]
From Theorem \ref{reg_harm}, we know that 
$|\grad^2 Z(x,s)|_{g(s)} \leq c(c_1,n)s^{-1/2} $ and hence 
  \begin{eqnarray}
  && |Z(x,s)-Z(x,t)| \leq c(c_1,n)\sqrt{s-t}  \label{initialc0est}
  \end{eqnarray} 
 for all $s \in [t,S(n,c_1)] \cap[0,T),$ $x \in  B_{d_0}(m_0,8 ) \subseteq B_{g(s)}(m_0,9) \subseteq  B_{d_0}(m_0,10)$.
We show the rest of the estimates hold for arbitrary  $r \in [2t,S(n,c_1,\al_0)] \cap [0,T/2],$  $x \in B_{d_0}(m_0,4),$  if $S(n,c_1,\al_0)$ is chosen to be small enough. First we scale the solution to the Ricci-harmonic map heat flow and the metric $g(t)$ by $1/ \sqrt{r}$ respectively $1/ r$:
$$\ti Z(z, s):=  \frac{1}{\sqrt r} Z(z, s r)\quad \text{and}\quad \ti g(\cdot, s )= \frac{1}{r}g(\cdot, r s)$$
for $z \in  B_{d_0}(m_0, 8 )$. Then the solution is defined for $s \in [\ti t := t/r, \ti T:= T/r]$  on $B_{\ti d_0}( x,  1/\sqrt{r} )$ for any $x  \in B_{d_0}(m_0, 7 )$   where $\ti t  = t/r \leq 1/2$, since $r \geq 2t$ and  $\ti T \geq 2,$ since $r \leq T/ 2$,  and the radius $V:= 1/ {\sqrt{r}}$ satisfies 
$ V \geq 1/ \sqrt{S(n,c_1,\al_0)}$.  Since before scaling, we have $|\Rm(\cdot,s)|\leq \ep^2_0/s$, after scaling we still have $|\widetilde \Rm(\cdot,s)|\leq \ep^2_0/s$ on $B_{\ti d_0}( x,V)$ for all 
$s \in [\ti t,2].$
The time $r$ has scaled to the time $1$.
The property \eqref{initprop} scales to 
\begin{equation}\begin{split}
|\ti Z(z,\ti t) -\ti Z(y,\ti t)| & \in \Big( (1-\ep_0)\, \ti d_{\ti t}(z,y) - \ep_0 \sqrt{\ti t}, (1+\ep_0)\,\ti d_{\ti t}(z,y)  +\ep_0 \sqrt{\ti t}\Big)  \label{tiZ_initial} \\
& \subseteq  \big( (1-\ep_0)\, \ti d_{\ti t}(z,y) - \ep_0 , (1+\ep_0)\, \ti d_{\ti t}(z,y) + \ep_0 \big)
\end{split}
\end{equation}
for all $z,y \in  B_{\ti d_0}( x,V)$  since $\ti t = t/r \leq 1/2$. The inequality \eqref{initialc0est} scales to
 \begin{equation}
 |\ti Z(z,  \ti s)- \ti Z(z, \ti t)| \leq c(c_1,n)\sqrt{\ti s-\ti t}  \label{scaledc0est}
  \end{equation} 
 for all $\ti s \in [\ti t,2],$ $z \in B_{\ti d_0}(x,V ),$ and the gradient estimate,
 is also scale invariant:   $|\ti \grad \ti Z(\cdot,\ti s)|_{\ti g(\ti s)} \leq c_1$ still holds on  $B_{\ti d_0}(x,V )$ for all $\ti s \in [\ti t,2]$.
We also have
$$|\ti Z(\cdot,\ti \si)-Z(\cdot,\ti t)|\leq c(c_1,n) \sqrt{\si}$$ on $B_{\ti d_0}(x, V)$, due to 
\eqref{scaledc0est}, for   $\ti \si := \si +\ti t,$  $\si \in (0,1)$.  Hence 
\begin{equation*}
\begin{split}
|\ti Z(z,\ti \si) - \ti Z(y,\ti \si)| &\geq    |\ti Z(z,\ti t) - \ti Z(y,\ti t)| - |\ti Z(z,\ti t) -\ti Z(z,\ti \si) | - | \ti Z(y,\ti t) -\ti Z(y,\ti \si) | \\
&\geq (1- \ep_0)\, \ti d_{\ti t}(z,y) -  \ep_0 -  c(c_1,n)\sqrt{\si} \\
&\geq (1- \si )\, \ti d_{\ti \si}(z,y) - 2c(c_1,n)\sqrt{\si}  
\end{split}
\end{equation*}
for $\si $ fixed and $\ep_0 \leq \si^2 $, 
and similarly, 
\begin{equation*}
|\ti Z(z,\ti \si) - \ti Z(y,\ti \si)|  \leq   (1+\si )\, \ti d_{\ti \si}(z,y) + 2c(c_1,n)\sqrt{\si} 
\end{equation*}
for $z,y  \in B_{\ti d_0}(x,V)$ if $\ep_0 \leq \si^2$.
That is 
$\ti Z(\cdot, \ti \si) $ is an  $\al^2$ almost isometry on $B_{\ti d_0}(x,V)$ if we choose $\si = \al^{8}$. At this point we fix $\al := \al(n,c_1,\al^3_0)$ where $\al$ is the function appearing in the statement of Lemma \ref{BasicLemma} and we set $\si := \al^{8}$. Without loss of generality, $\al \leq \al_0< c(c_1,n),$ and $(\al_0)^{-1}\geq 2c(c_1,n)$ for any given $c(c_1,n) \geq 1.$
   We  also still assume $\ep_0 \leq \si^2 = \al^{16}$, so that the previous conclusion,  $\ti Z(\cdot, \ti \si) $ is an  $\al^2$-almost isometry,  and hence certainly an $\al$-almost isometry, on $B_{\ti d_0}(x,V)$, holds, as explained above.

The curvature estimate,  $|\widetilde \Rm (\cdot,s)| \leq \ep^2_0/s $ for all  $s \in  [0,2] \supseteq [\ti t,2]$, holds, as do  the scaled distance estimates,
\begin{equation}\label{tidistest}
\ti d_{ \ell} +\ep_0\sqrt{s-\ell}  \geq \ti d_{s}  \geq \ti d_{\ell} -\ep_0 \sqrt{s -\ell} 
\end{equation}
  on $B_{\ti d_0}(x, V) \Subset M,$ for all $0\leq \ell \leq s  \in [0,2].$ 
The estimates of Shi imply, as explained in Lemma \ref{RicciFlowBasics}, that 
at time $\ti \si:= \ti t + \si =\ti t + \al^8,$ 
\begin{equation*}
|\widetilde \Riem|(\cdot,\ti \si) + |\ti \grad \widetilde\Riem|^2(\cdot,\ti \si) + \ldots +|\ti \grad^k \widetilde \Riem|  (\cdot,\ti \si)\leq  \frac{\be(k,n,\ep_0)}{\si^{k +2}}= \frac{\be(k,n,\ep_0)}{\al^{8k + 16}} \label{almost_flat_begin}
\end{equation*}
on $B_{\ti d_0}( x, 2V/3)$ 
 where  $\be(k,n,\ep_0) \to 0$ as $\ep_0 \to 0$ for fixed $k$ and $n$. In particular 
\begin{equation}
\sup_{ B_{\ti d_0}( x, 2V/3)}|\widetilde \Riem|(\cdot,\ti \si) + |\ti \grad \widetilde\Riem|^2(\cdot,\ti \si) + \ldots + |\ti \grad^k \widetilde \Riem|  (\cdot,\ti \si)\to 0\label{almost_flat}
\end{equation} 
as $\ep_0 \to 0$ for fixed $c_1,k,n,\al_0.$
 Without loss of generality, $L = 1/\sqrt{S(n,c_1,\al_0)}  \geq 10/\al$
 and hence $B_{\ti d_s}(x, 2\al^{-1})  \subseteq  B_{\ti d_0}( x, L/2)  \subseteq B_{\ti d_0}( x,2V/3)$ for all $s \in [0,2]$.

By  \eqref{tidistest}, \eqref{almost_flat},  and \eqref{tiZ_initial},
 we see  using Lemma \ref{DGLemma} with $h=\ti g(\ti t)$, and $g= \ti g(\ti \si)$, $L = \al^{-1}$,  that $( B_{\ti d_0}(x, \al^{-1}),\ti g(\ti \si))$ is $\al$ close in the $C^k$-norm to a Euclidean ball with the standard metric  in the Cheeger-Gromov sense (that is up to smooth diffeomorphisms), if $\ep_0$ is small enough.

Hence there are geodesic coordinates $\phi$ on the ball  $(B_{\ti d_0}(x,\al^{-1}),\ti g(\ti \si))$ such that the metric $\ti g(\ti \si)$ written in these coordinates is $\al$ close to
 $\de$  in the $C^k$ norm, if we keep $\si$ fixed and choose $\ep_0$ small enough.
 Using  \eqref{almost_flat}, and the evolution equation $\partt \ti g = -2\Rc(\ti g)$ in the coordinates $\phi$, we see that 
 the evolving metric $ h(\cdot) = \phi_*( \ti g(\cdot))$  in these coordinates is also, without loss of generality, $\al$ close to $\de$ for $t \in [\ti \si, 2]$ in the $C^2$ norm.

Using the above, we see that  $$ G(\cdot,t):= (\ti Z(\cdot,t+\ti \si) - \ti Z(x,\ti \si)) \of (\phi)^{-1}(\cdot)$$ defined on $\B_{\al^{-1}}(0)\times [0, 3/2]$ sends $0$ to $0$ at $t=0$,  is Lipschitz with respect to $\de$ with Lipschitz constant $2c(c_1,n)$ and is 
an $\al$ almost isometry at time $0$.
Lemma \ref{BasicLemma} is then applicable to the function
$G$  and tells us $G(\cdot,s)$ is an $\al_0^3$ almost isometry at $s= 1- \ti \si$  on a ball of radius $(\al_0)^{-3}$ and  that the   inequalities \eqref{BasicLemmaIneq} hold.
Hence, 
\begin{equation*}
\begin{split}
|d \ti Z(v)(1)| &\in \big( (1-\al^3_0)|v|_{\ti g(1)}, (1+\al^3_0)|v|_{\ti g(1)}\big), \\
|\ti Z(z,1) - \ti Z(y,1)| &\in \big( (1-  \al_0^{3} ) d_{\ti g(1)}(z,y), (1+\al_0^{3})d_{\ti g(1)}(z,y)  \big),
\end{split}
\end{equation*}
and  
$$|\ti Z(z,\ti t) - \ti Z(z,1)| \leq \al^3_0 \sqrt{1-\ti t}$$ 
for all $z,y \in 
B_{\ti d_0}(x,\al_0^{-3})$ for all $v \in T_z M$.
This scales back to 
\begin{equation*}
\begin{split}
|d Z(v)(r)| &\in \big( (1-\al^3_0)|v|_{g(r)}, (1+\al^3_0)|v|_{ g(r)}\big), \\
|Z(z,r) -Z(y,r)| &\in \big( (1-  \al_0^{3} )\, d_{g(r)}(z,y), (1+\al_0^{3})\, d_{g(r)}(z,y)  \big),
\end{split}
\end{equation*}
and 
$$|Z(z, t) -  Z(z,r)| \leq \al^3_0\sqrt{r-t}$$
for all $z,y \in 
B_{d_0}(x,\sqrt{r}/\al^3_0)$ for all $v \in T_z M$.

For  $z \in B_{d_0}(x,\sqrt{r}/(2\al^3_0))$ and $y \in  ( B_{d_0}(x,\sqrt{r}/\al^3_0 ))^c \cap ( B_{d_0}(m_0,6)),$ we show that 
the property \eqref{flow_bilip}  also holds. For such $z,y$, we have,  
 $d_{0}(z,y)  \geq \sqrt r /(2 \al^3_0)$ and hence 
\begin{equation*}
\begin{split}
  |Z(z,r)-  Z(y,r)| & \geq   |Z(z,t) -Z(y,t)|  -c(c_1,n)\sqrt{r-t} , \\
  & \geq (1-\ep_0)\, d_t(z,y) - \ep_0\sqrt{t} - c(c_1,n) \sqrt{r} , \\
& \geq (1 - \ep_0)\, d_0(z,y) +  (1-\ep_0)( d_t(z,y) -d_0(z,y)) - 2c(c_1,n) \sqrt{r} , \\
&  \geq(1 - \ep_0)\, d_0(z,y) -\ep_0\sqrt{r} - 2c(c_1,n) \sqrt{r} , \\
& \geq (1 - \al_0^{2})\, d_{0}(z,y) - 3c(c_1,n)   \sqrt{r} \\
& \geq (1- 10\al_0^2 )\, d_{0}(z,y)  + 9 \al_0^2 d_0(z,y) - 3c(c_1,n)   \sqrt{r} \\ 
& \geq (1- 10\al_0^2 )\, d_{0}(z,y)  + 9\al_0^2 \frac{ \sqrt{r} }{ 2\al_0^3}  - 3c(c_1,n)   \sqrt{r} \\ 
& \geq (1- 10\al_0^2 )\, d_{0}(z,y)  + 9c(c_1,n) \sqrt{r}   - 3c(c_1,n)   \sqrt{r} \\ 
& \geq (1- 10\al_0^2  )\, d_{0}(z,y) +   \sqrt{r}  \\
& \geq  (1- 10\al_0^2  )\, d_{r}(z,y) -\ep_0 \sqrt{r} +  \sqrt{r} \\
& \geq (1-\al_0)\, d_{r}(z,y) 
\end{split}
\end{equation*} 
as required, where the first inequality follows from \eqref{initialc0est}, the second from the condition \eqref{c},  the seventh and eighth from the fact that $d_{0}(z,y) \geq \sqrt{r} /( 2 \al^3_0) \geq c(c_1,n) \sqrt{r}/ \al^2_0$ and we have used $\ep_0 < \al_0^{4},$ $c(c_1,n) \geq 1,$ the distance estimates \eqref{distest}, and $\al_0 \leq 1 /c(c_1,n)$ freely.

It remains to show the property that 
$$\B_{1}(Z(s)(m_0))  \subseteq D_s \subseteq \B_{2}(Z(s)(m_0))$$ for $D_s:= Z(s)(B_{d_0}(m_0,3/2)) \subseteq \R^n,$ 
for $s\in (2t,S]$, $s \leq T/2$, $Z(s):= Z(\cdot,s)$.
Observe that 
$$Z(s): B_{d_0}(m_0,3/2) \to \R^n$$
 is smooth and satisfies \eqref{flow_gradest}. In particular,  $Z(s)$ is a local diffeomorphism, due to the Inverse Function Theorem, and hence 
$\B_{r}(Z(s)(m_0)) \subset D_s$ for some maximal $r>0.$  Let $p_i := Z(s)(x_i) \in  \B_{r}(Z(s)(m_0)) \cap D_s,$   such that $p_i \to p  \in \boundary  \B_{r}(Z(s)(m_0)) $ where $p \notin D_s$. Assume that $r\leq 5/4$. Then $ x_i \in B_{d_0}( m_0, 5/4 +C(n)\al )$. 
After taking a subsequence, if necessary, $x_i \to x \in B_{d_0}(m_0,5/4+C(n)\al)),$ and consequently, $p_i = Z(s)(x_i)  \to Z(s)(x) = p \in D_s,$ as $i\to \infty$  which is a contradiction to the definition of $r>0$, if $r\leq 5/4.$ Hence $r\geq 5/4,$ which implies
 $ \B_{5/4}(Z(s)(m_0)) \subset D_s,$  as claimed (the second inclusion follows immediately from the Bi-Lipschitz property).
 \end{proof}
 
With the help of the previous theorem, we now show that it is possible to construct a solution to the $\de$-Ricci-DeTurck flow coming out of $\ti d_0:= (F_0)_{*}d_0$ using the harmonic map heat flow, if we assume that  \eqref{chat} is satisfied. 
First we show that by slightly mollifying the distance coordinates at time $t$, we obtain maps which satisfy \eqref{c}.

\begin{lemma}\label{mollified_dist}
Let $(M,g(t))_{t\in [0,T]}$ be a solution to Ricci flow satisfying \eqref{tildea}, \eqref{b}, for some  $R \geq \be^2(n) \ep_0^2 +200$ 
and let $d_0$ be as defined in \eqref{definitiond_0}. Assume that there are points $a_1, \ldots, a_n$ such that 
 $F_0: B_{d_0}(x_0,R) \to \R^n,$  $F_0(x) := (d_0(a_1,x), \ldots, d_0(a_n,x))$ satisfies 
 \eqref{chat} on $B_{d_0}(x_0,100)$ , and let
$F_t:B_{d_0}(x_0,R) \to \R^n$ be given by
$$F_t(x) = (d_t(a_1,x), \ldots, d_t(a_n,x)).$$ 
Then by mollifying    $F_t$ at an  appropriately small scale, as explained in the proof, we obtain a map $\moll{F_t}:B_{d_0}(x_0,R) \to \R^n$ which is smooth and satisfies $|\grad^{g(t)} \moll{F_{t}}|_{g(t)} \leq c(n)$ as well as
\eqref{c} on $B_{d_0}(x_0,50)$  (with $\ep_0$ replaced by $2\ep_0$), provided $t\leq \hat T(\ep_0,R)$.
\end{lemma}
\begin{proof}
As already noted, 
$$F_t|_{B_{d_0}(x_0,50)}: B_{d_0}(x_0,50) \to \R^n$$
 satisfies  \eqref{c}  in view of the distance estimates \eqref{distest}, if $t \leq \hat T(\ep_0,R)$.  Also, it is well known, that the Lipschitz norm of any map $F_t$ as defined above may be estimated by a constant depending only on $n$:  
$$\sup_{x \neq y \in B_{d_0}(x_0,R)} \frac{|F_t(x) -F_t(y)|}{d_t(x,y)} \leq  c(n),$$ 
in view of the triangle inequality. Hence, by  mollifying the map $F_t$ at an appropriately small scale, we obtain a map $\moll{F_{t}}:B_{d_0}(x_0,50) \to \R^n$ which is  smooth  and satisfies  \eqref{c} (with $\ep_0$ replaced by $2\ep_0$)  and  $|\grad \moll{F_{t}}|_{g(t)} \leq c(n)$.
\end{proof}

 \smallskip

\begin{thm}\label{RicciDeTurck}
Let $(M,g(t))_{t\in [0,T]}$ be a solution to Ricci flow satisfying  \eqref{tildea}, \eqref{b}  for an 
$R \geq \be^2(n) \ep_0^2 +200$ 
and let $d_0$ be as defined in \eqref{definitiond_0}. Assume that there are points $a_1, \ldots, a_n$ such that 
 $F_0: B_{d_0}(x_0,R) \to \R^n,$  $F_0(x) := (d_0(a_1,x), \ldots, d_0(a_n,x))$ satisfies 
 \eqref{chat} on $B_{d_0}(x_0,100)$,
  and let  $\moll{F}_{t_i}$ be the corresponding  mollified functions from Lemma \ref{mollified_dist} for any sequence of times $t_i>0$ with $t_i \to 0$ as $i\to \infty$.
Let $\hat S := \min(S(n,\alpha_0)$ and
$$Z_{t_i}:B_{d_0}(x_0,100) \times [t_i,\hat S,T/2)] \to \R^n$$ 
be the Dirichlet solution to the Ricci-harmonic map heat flow  with boundary and initial values given by $Z_{t_i}(\cdot,s)|_{\boundary B_{d_0}(x_0,100)} = \moll{F}_{t_i}(\cdot),$ for all $s \in [t_i,\hat S]$  and $Z_{t_i}(\cdot,t_i)  = \moll{F}_{t_i}(\cdot)$.
Then, after taking a subsequence in $i$, the maps  
$$Z_{t_i}(s): B_{d_0}(x_0,3/2) \to D_{s,i}:= Z_{t_i}(s)(B_{d_0}(x_0,3/2)) \subseteq \R^n$$  
are homeomorphisms for all $s \in [ 2t_i, \hat S] $, with $\B_{1}( F_0(x_0)) \subseteq D_{s,i}$ and 
$(Z_{t_i})_* g \to \ti g$ smoothly, as $i\to \infty$ on compact subsets of $ \B_{1}( F_0(x_0)) \times (0,\hat S],$  where ${\ti g(s)}_{s \in (0,\hat S]}$ is a smooth family of metrics which solve the $\de$-Ricci-DeTurck flow and 
\begin{equation}
 (1-\al_0) \de \leq \ti g(s)  \leq (1+\al_0) \de \label{metricin}
 \end{equation}
 for all $s \in(0,\hat S)$, provided $\ep_0 = \ep_0(\al_0,n)>0$ from \eqref{a}, and \eqref{chat} are small enough.
 The metric $\ti d(t): = d(\ti g(t))$ satisfies, 
 $ \ti d(t) \to \ti d_0:= (F_0)_*d_0$ uniformly on $  \B_{1}( F_0(x_0))$ as $t \to 0.$ 
\end{thm} 
\begin{rmk}
Examination of the proof of Theorem \ref{RicciDeTurck} shows that: \\[1ex]
(i) We can remove condition \eqref{b} if we assume that the estimates \eqref{distest} are satisfied.\\[1ex]
(ii) If we remove  condition  \eqref{chat} and replace it by the assumption:
there exists a sequence of times $t_i>0$ with $t_i \to 0$ as $i\to \infty$, and maps
$\hat F_{t_i}: B_{d_0}(x_0,100)  \to \R^n$ each of which satisfies  \eqref{c}, $\sup_{i\in \N } |\hat F_{t_i}(x_0)| < \infty,$  and $|\grad^{g(t_i)} \hat F_{t_i}|_{g(t_i)} \leq c_1$ then we can use these $\hat F_{t_i}$ in the above, instead of the slightly mollified distance functions,   and the conclusions of the theorem still hold for $s \leq  \hat S:= \min(S(n,c_1,\al_0),T/2)$ if the $\ep_0 = \ep_0(n,c_1,\al_0)$ appearing in \eqref{a} and \eqref{c} is small enough. In this case,  $F_0$  is the uniform $C^0$ limit of a subsequence of the $F_{t_i}$ as $i\to \infty$  and satisfies \eqref{bilipproperty}. The existence of such an $F_0$  is always guaranteed in this setting,  as explained directly after the introduction of the condition \eqref{c}.

\end{rmk}
\begin{proof}
Theorem \ref{flowthm} tells us that the maps  $Z_{t_i}(s): B_{d_0}(x_0,3/2) \to D_{s,i}$ are homeomorphisms with 
$\B_{5/4}( Z_{t_i}(s)(x_0)) \subseteq D_{s,i}$  for $s \in [2t_i,\hat S]$.
Hence $$ \B_{1}(F_0(x_0)) \subseteq  \B_{5/4}( Z_{t_i}(s)(x_0)) \subseteq D_{s,i},$$  for $s \in [2t_i,\hat S],$ in view of \eqref{flow_c0time}.
We define $\ti g_i(s): = (Z_{t_i})_{*}g(s)$ for $s \in [2t_i,\hat S]$ on $\B_{\frac 3 2 }(F_{t_i}(x_0))$. This is well defined in view of Theorem \ref{flowthm}. Then $\ti g_i $ is a solution to the $\de$-Ricci-DeTurck flow on $\B_{3/2}(F_{t_i}(x_0))$ (see \cite[Chapter 6]{HamFor} for instance) and satisfies the metric inequalities 
\eqref{metricin} for all $s \in (2t_i, \hat S)$ in view of Corollary \ref{corflowthm}.
Using \cite[Lemma 4.2]{MilesC0paper} we see that
\begin{equation*}
|D^j \ti g_i(s)| \leq \frac{c(j,n)}{(s-2t_i)^{p(j)}}
\end{equation*}
for all $j\in \N$, for all $s\in (2t_i,\hat S)$ on $\B_{1}(F_{t_i}(0))$. Taking a subsequence in $i$  we obtain the desired solution $\ti g(s)_{s\in (0,\hat S)}$ on $\B_{1}(F_0(0))$ with 
\begin{equation*}
|D^j \ti g(s)| \leq \frac{c(j,n)}{s^{p(j)}}.
\end{equation*}
The $Z_{t_i}$ all satisfy the estimates stated in the conclusions of Theorem \ref{flowthm}, and so there is a uniform $C^{1,\al}$   limit map
$Z: B_{d_0}(x_0,2) \times (0,\hat S) \to \R^n$, in view of the Theorem of Arzel\`a-Ascoli. Furthermore, $Z(s)= Z(\cdot,s)$ satisfies   
$$|Z(s) -F_0| \leq c_1\sqrt{s}$$ for $s \in (0,\hat S)$ in view of the estimate \eqref{flow_c0time}. 
Let $v, w \in \B_1(0)$ be arbitrary, and $x,y$ the corresponding points in $B_{d_0}(x_0,2)$ at time $s$, that is the unique points
$x,y$ with $Z(s)(x) = v,Z(s)(y) = w$. Then 
\begin{equation*}
\begin{split}
\ti d_s(v,w) & = d_s(x,y)  \leq  d_0(x,y) + \ep_0 \sqrt{s}  = \ti d_0(F_0(x),F_0(y))  + \ep_0 \sqrt{s}  \\
& \leq \ti d_0(Z(s)(x), Z(s)(y)) + \be(s) + \ep_0 \sqrt{s}  = \ti d_0(v,w) + \be(s) + \ep_0 \sqrt{s}, 
  \end{split}
  \end{equation*}
  where $\be(s) \to 0$ as $s \to 0$. Here, we have used the fact that $\ti d_0:= (F_0)_*(d_0)$ is continuous, and hence uniformly continuous on 
  $\overline{\B_1(0)}$, and that 
  $$\sup_{B_{d_0}(x_0,2)} |Z(s)(\cdot)  -F_0(\cdot) | \leq c_1 \sqrt{s}$$ 
  for all $s \in (0,\hat S)$ in the above.  The continuity of $\ti d_0 := (F_0)_*d_0$ with respect to the norm in $\R^n$ follows from the fact that $\ti d_0$ is a metric, equivalent to the standard metric  on $\R^n$ in view of the property (\ref{bilipproperty}). Similarly, $\ti d_s(v,w) \geq \ti d_0(v,w)  - \be(s) - \ep_0 \sqrt{s},$ as required. 
\end{proof}

\section{Ricci-harmonic map heat flow in the continuous setting}\label{continuous_solutions}

We now assume, in addition to the assumptions \eqref{a}, \eqref{b} and \eqref{chat} of the previous chapter, more regularity on $d_0$ and $\ti d_0$.
Namely, we assume that $\ti d_0$ is generated by a continuous Riemannian metric $\ti g_0$ on $\B_{1}(0)$. This assumption will guarantee for all $\ep>0$ the existence of local maps defined on balls of radius $r(\ep)$,  which are $1+\ep$ Bi-Lipschitz maps at $t=0$. 
We explain this in the following Lemma.
\begin{lemma}\label{continuouslemma}
Let $ (X,d)$ be a metric space,  $B_{d}(y_0,10) \Subset X$ and 
$F:B_{d}(y_0,1) \to \R^n$ be  a $1 + \ep_0$  Bi-Lipschitz homeomorphism  with $F(y_0) = 0$, and assume that $\ti d = (F)_*d$ is generated on $\B_{1/4}(0)$  by a continuous Riemannian metric $\ti g$, which is defined on $\B_{1}(0)$. 
Then for all $\ep>0$,  there exists an $r>0$ such that for all $p \in B_{d}(y_0,1/8)$ there exists a linear transformation, $A=A(p)$,  $A:\B_{r}(F(p)) \to \R^n$  with $|A-Id|_{C^0} \leq 2$ such that $ \hat F := A \circ F$ satisfies
\begin{equation*}
 (1-\ep)|\hat F(y) -\hat F(q)| \leq d(y,q) \leq (1+\ep) |\hat F(y) -\hat F(q)| 
\end{equation*}
 for all $y,q \in B_{d}(p,r/2 ),$  and
  $\vol(B_{d}(p,s)) \in ( (1-\ep)^{n} \omega_n s^n, (1+\ep)^{n} \omega_n s^n)$ for all $s \leq r,$ for all such $p$.
\end{lemma}
 \begin{proof}
The continuity of $\ti g$ means: for any $\ep>0$ and any $x \in \B_{1}(0)$ we can find an $r>0$ and a 
 linear transformation,
$A: \B_{r}(x) \to A(B_r(x))$ with $|A-Id|_{C^0(\B_r(x))} \leq c(n) \ep_0$ so that $\hat g := A_{*}(\ti g)$ satisfies 
$|\hat g - \de|_{C^0(B_r(x))} \leq \ep,$ and  $\hat g (x) = \de$. For the distance  $\hat d := A_{*} \ti d$  this means 
\begin{equation*}
(1-\ep) |z-w|  \leq \hat d(z,w) \leq (1+\ep)|z-w| 
\end{equation*}
for all $z,w \in \B_{r/2}(A(x)).$
Returning to the original domain, we see that this means
\begin{equation*}
 (1-\ep)|\hat F(y) -\hat F(q)| \leq d(y,q) \leq (1+\ep) |\hat F(y) -\hat F(q)| 
\end{equation*}
 for all $y,q \in B_{d}(p,r/2 )$ where $\hat F(p) = A(x)$, and
 $\hat F= A \circ F$.
 This means in particular in view of the existence of the $1 +\ep$ Bi-Lipschitz map $\hat F$, that
 $(1+c(n)\ep)^{n} \omega_n s^n \geq \vol(B_{d}(p,s)) \geq (1-c(n)\ep)^{n} \omega_n s^n$ for all $s \leq r$.
 \end{proof}
This implies that we can replace condition \eqref{a} by condition 
\begin{equation} \label{ahat}\tag{$\hat{\rm a}$}
\begin{split}
& |\Rm(\cdot,t)| \leq \frac{\ep(t)}{t} \mbox{ on } B_{d_0}\left(x_0, 1/10 \right) \ \ \mbox{ 
for  all}   \ \ t \in (0,1) \mbox{ where }   \\
& \ep:[0,1] \to [0,1] \mbox{  is a continuous non-decreasing function  with }  \ep(0) = 0,
\end{split}
\end{equation}
as we show in the following Lemma.

\begin{lemma}\label{curvaturedecaylemma} 
Assume $(M,g(t))_{t\in (0,T)}$ is a solution to Ricci flow satisfying \eqref{tildea} and \eqref{b} for some $R \geq \be^2(n)\ep_0^2 + 200$, and assume $F_0:B_{d_0}(x_0,1) \to \R^n $ 
is a Bi-Lipschitz map and that $\ti d_0 = (F_0)_*d_0$ and $F_0$  satisfy the assumptions of Lemma \ref{continuouslemma}, where $d_0$ is defined by \eqref{definitiond_0}. Then \eqref{ahat}  holds on $ B_{d_0}\left(x_0, 1/10\right). $
\end{lemma}
\begin{proof}
Let $\si >0$ be given, and assume, that there are $t_i \to 0$ and $p_i \in B_{d_0}(x_0, 1/100)$ with $|\Rm(p_i,t_i)| = \si/t_i$. We scale the solution $(g(t))_{t\in(0,T)}$ so that the time $t_i$ scales to time $1$, i.e.~we define a sequence of solutions to Ricci flow as follows: 
$$g_i(s):= t_i^{-1}g(t_i s)$$ 
for $s\in(0,T/t_i)$.

The distance estimates \eqref{distest}, the estimates of Shi, \eqref{shi_est}, hold, and $\vol( B_{g_i(s)}(p_i,1)) \geq \omega_n/2$ for all $s \in (0,100)$ in view of Corollary \ref{corflowthm}. Hence, after taking a subsequence, we obtain a smooth solution $(\Omega,\ell(t),p_0)_{t\in  (0,10]}$ with $|\Rm(p_0,1)| =\si,$ 
$\Rc \geq 0$ everywhere, and $|\Rm(\cdot,t)| \leq \ep^2_0/t$ everywhere.  Furthermore, writing $d_0(i) := t_i^{-1/2}d_0$, we see, in view of the distance estimates,  that  $d(\ell(t)) \to \hat d_0$ with $t \to 0$, where  
$( B_{d_0(i)}(p_i, \frac{1}{\sqrt{t_i}} ), d_0(i) , p_i) \to (\Omega, \hat d_0, p_0)$ in the Gromov-Hausdorff sense, as $i\to \infty$. But $(\Omega, \hat d_0)$  must be isometric to $(\R^n,\de)$ since there are $1+ \ep$ Bi-Lipschitz maps $F(i):(B_{d_0(i)}( p_i, r(\ep)/\sqrt t_i),d_0(i)) \to \R^n$   for all $\ep>0$, in view of Lemma \ref{continuouslemma}.
Using a similar argument to the one used in \cite{Yokota} and \cite{SchulzeSimon}, we see that the   asymptotic volume ratio of $(\Omega,\ell(t))$ must be $\omega_n$, as we now explain. Take a sequence $R_i \to \infty$ at time $s>0$. Scale $\ell_i(s) := R^{-2}_i\ell(sR_i^2)$. 
Since  $(\Omega, R_i^{-1} \hat d_0)$ is isometric to $(\R^n,\de)$
for all $i \in \N,$ we must have 
$(B_{\hat \ell_i}(x_0,1),\hat \ell_i )  \to  (\B_1(0),\de)$ in the Gromov-Hausdorff sense as $i\to \infty,$ for $\hat \ell_i := \ell_i(1/R_i^2 )$ ($=R_i^{-2} \ell(1)$)
in view of the scaled distance estimates \eqref{distest}.
Hence, $\vol(B_{\hat \ell_i}(x_0,1),\hat \ell_i) \to  \omega_n$ as $i\to \infty$, in view of the Theorem of Cheeger-Colding on volume convergence (Theorem 5.9 of \cite{cheeger_colding_vol_cgnce}).
But this means that the asymptotic volume ratio of $(\Omega,\ell(s),p_0)$ is $\omega_n$, and hence, $(\Omega,\ell(s),p_0)$ is isometric to $(\R^n,\de)$,  in view of the Bishop-Gromov comparison principle (the case of equality). This  contradicts the fact that $|\Rm(p_0,1)| =\si.$ 
\end{proof}

Note that we obtain the  better distance estimates in the setting of this Lemma,
\begin{equation}\label{distest2}
\begin{split}
 &d_r + \ep(t) \sqrt{t-r}  \geq d_t  \geq d_r - \ep(t)\sqrt{t-r} \mbox{ for all } t\in [r,1) \\
&\mbox{ on } B_{d_0}(x_0, 1/20) \subseteq B_{g(s)}(x_0,1) \Subset M  
\end{split}
\end{equation}
for all $r\geq 0$, where $\ep$ is without loss of generality, the same function appearing in the condition \eqref{ahat}.

\begin{thm}\label{continuous_thm}
Assume $(M,g(t))_{t\in (0,T]}$ is a solution to Ricci flow satisfying \eqref{tildea} and \eqref{b}, and  that there are points $a_1, \ldots, a_n$ such that 
 $F_0: B_{d_0}(x_0,R) \to \R^n,$  $F_0(\cdot) := (d_0(a_1,\cdot), \ldots, d_0(a_n,\cdot))$ satisfies 
 \eqref{chat} on  $B_{d_0}(x_0,100),$ and  $\ti d_0 = (F_0)_*d_0$ and $F_0$ satisfy the assumptions of Lemma \ref{continuouslemma}, where $d_0$ is as defined in \eqref{definitiond_0}.
Then the solution \\$(\B_{1/2}(0),\ti g(s)) _{s\in (0,\min{T,S(n,\al_0,c_1)}]}$ to $\de$-Ricci-DeTurck flow   constructed in Theorem \ref{RicciDeTurck} satisfies $|\ti g(s) -\ti g_0|_{C^0(\B_{1/20}(0))} \to 0$ as $s \to 0$.
\end{thm}
\begin{proof}
Using Lemma \ref{continuouslemma} we see the following: for any $\ep>0$ and any $p_0 \in B_{d_0}( x_0, 1/20 ) $ we can find an $r>0$ and a 
 linear transformation $A: \R^n \to \R^n$ ,  such that 
\begin{equation}
 (1-\ep)|\hat F_0(y) -\hat F_0(q)| \leq d_0(y,q) \leq (1+\ep) |\hat F_0(y) -\hat F_0(q)| \label{inbetween}
\end{equation}
 for all $y,q   \in B_{d_0}(p_0,r),$  
 for $\hat F_0= A \circ F_0$. We define  $ z_0 := F_0(p_0)$ and $\hat z_0 := \hat F_0(p_0)$. 
 Now since $A$ is a linear transformation with $|A- Id|_{C^0} \leq 2,$ and  \eqref{distest2} holds, we have $|\hat F_t - \hat F_0|_{C^{0}(B_{d_0}(x_0, \frac{1}{20}))}  \leq  \ep(t) \sqrt{t}$ for $\hat F_t = A \circ F_t$, and   hence 
 \begin{equation*}
(1-\ep)  d_t(v,q) - \ep \sqrt{t}   \leq   |\hat F_t(y) -\hat F_t(q)| \leq (1+\ep) d_t(v,q) + \ep \sqrt{t}  \label{inbetween2}
 \end{equation*}
 on $B_{d_0}(p_0,r)$ for all $t\leq T(\ep),$ where we have also used  \eqref{inbetween}.
 Let $Z_{t_i}:B_{d_0}(x_0, 1/2) \times [t_i,S(n,\al_0) ) \to \R^n$ be the solutions to Ricci-harmonic map heat flow 
 constructed in Theorem \ref{RicciDeTurck}. Then $\hat Z_{t_i} = A \circ Z_{t_i}$ is also a solution to Ricci-harmonic map heat flow.
 Using the regularity theorem, Theorem \ref{flowthm}, we see that we must have 
 \begin{gather*}
   (1-\si(\ep))d_{s}(z,w) \leq  | \hat Z_{t_i}(z,s) -  \hat Z_{t_i}(w,s) |  \leq   (1+\si(\ep))d_{s}(z,w)\\
   |\grad^{g(s)} \hat Z_{t_i}(v)|  \in \big( (1 - \si(\ep)) |v|_{g(s)} , (1 + \si(\ep)) |v|_{g(s)}\big) 
   \end{gather*}
   for all $z,w \in B_{d_0}(p_0,r/5)$ for all $v \in T_z B_{d_0}(p_0,r/5)$, for all $s \in (2t_i,S(n,\ep))$ where 
   $\si(\ep) \to 0$ as $ \ep \to 0$.
Hence Corollary \ref{corflowthm} tells us that the push forward $\hat g_i(s): = (\hat Z_i)_*(g(s))_{s \in  (2t_i,S(n,\ep))}$ satisfies $|\hat g_i(s) -\de|_{C^0(\B_{r/5}(\hat z_0))} \leq \si(\ep).$
Transforming back with $A^{-1}$ we see that this means
 $|\ti g_i(s) -\ti g_0|_{C^0(\B_{r/5}(z_0))} \leq \si(\ep)$ for all $s \in (2t_i,S(n,\ep))$, and hence
 $|\ti g(s) - \ti g_0|_{C^0(\B_{r/5}(z_0))} \leq \si(\ep)$ for all $s \in (0,S(n,\ep)).$ As $p_0 \in \B_{d_0}(x_0,\frac{1}{20})$ was arbitrary, we see by  letting $\ep \to  0$, that $|\ti g(s) - \ti g_0|_{C^0( \B_{1/20}(0))} \to 0$ as $s \to 0$, as required. \end{proof}

The estimates of the previous theorem allow us to give a proof of the second main theorem of the introduction:

\begin{proof}[Proof of Theorem \ref{continuity_of_solutions_intro}]
Let 
$a_1,\ldots, a_n \in B_{d_0}(x_0,r)$ be as in the statement of Theorem \ref{continuity_of_solutions_intro}. That is,
$F_0(\cdot):= (d_0(a_1,\cdot), \ldots, d_0(a_n,\cdot))$  is $1+\ep_0$ Bi-Lipschitz
on $B_{d_0}(x_0,5\ti r)$  for some $\ti r \leq r/5$. The metric 
 $\ti d_0(\ti x, \ti y):= d_0((F_0)^{-1}(\ti x), (F_0)^{-1}(\ti y))$ defined on 
$ B_{\ti r}(F_0(x_0))$ 
is generated by a continuous (with respect to the standard topology on $\R^n$) Riemannian metric $\ti g_0$ defined on 
$B_{ 4 \ti  r } (F_0(x_0)) \subseteq \R^n $. 
By scaling everything once, that is $\hat g(t) = \ti r^{-1} g(t)$ $\hat  d_0 = \ti r^{-1/2} d_0$,
we see that we are in the setting of Theorem   \ref{continuous_thm} (choosing $R = 1/\sqrt{\ti r}$ in conditions \eqref{a},\eqref{b},\eqref{chat}).  The conclusions of that theorem, when scaled back, imply the  conclusions of Theorem  \ref{continuity_of_solutions_intro}.
\end{proof}

  \section{Existence and estimates for the  Ricci-DeTurck Flow with $C^0$ boundary data }\label{Dirichlet_Section}
In this section we construct solutions $\ell$ to the Dirichlet problem for the $\de$-Ricci-DeTurck flow on a Euclidean ball, which are smooth up to the boundary at time zero, and have $C^0$ parabolic boundary values. These solutions are constructed as a limit of smooth solutions $ \ell_{\al}$ whose parabolic boundary values converge to those of $\ell$. 

 Recall that the $\de$-Ricci-DeTurck flow equation for a smooth family of metrics $\ell$ is given by (see  \cite[p.~15]{HamFor} and/or \cite[Lemma 2.1]{Shi})
\begin{equation}\label{eq:RicciDeTurckdelta}
\begin{split}
  \partt \ell_{ij}=\,\ell^{ab}\partial_a\partial_b \ell_{ij} +\tfrac12 \ell^{ab}\ell^{pq}\big(&\partial_i \ell_{pa}\partial_j\ell_{qb}
    +2\partial_a\ell_{jp}\partial_q\ell_{ib}\\
    &-2\partial_a\ell_{jp}
    \partial_b\ell_{iq}-2\partial_j\ell_{pa}\partial_b\ell_{iq}
    -2\partial_i\ell_{pa}\partial_b\ell_{jq}\big)\, .
\end{split}
\end{equation}
First we prove an estimate about the closeness of smooth solutions to $\de$ in the $C^0$ norm,  assuming $C^0$ closeness on the spatial boundary and a bound on the $C^2$ norm at time zero.
\begin{lemma}\label{C0Lemma}
Let $\ell$ be a $H^{2+ \al, 1+ \frac{\al}{2}}(\B_R(0) \times [0,T]) \cap C^0(\overline{\B_R(0)} \times [0,T])$ solution to  the $\de$-Ricci-DeTurck flow such that $\ell(\cdot,0) = \ell_0$, 
$\|\ell(\cdot,t) - \ell_0\|_{ L^\infty(\boundary \B_R(0))} \leq \be$ for $t\in [0,T]$ and
$\|\ell_0-\de\|_{C^2(\B_R(0))} \leq \varepsilon(n)$, where $\beta \leq \varepsilon(n)$. Then, 
 \begin{equation*}
 \begin{split}
 \|\ell(\cdot,t) - \de\|_{L^\infty(\B_R(0))} &\leq  c(n)\varepsilon(n),\quad\mbox{for all $t\in [0,T]$,}\\
 \|\ell(\cdot,t) - \ell_0\|_{L^\infty(\B_R(0))} &\leq  \beta +\varepsilon(n)t,\quad \mbox{for all $t\in [0,T]$.}
 \end{split}
 \end{equation*}
\end{lemma}
\begin{proof} We will denote in the following $| \cdot |$ all norms induced by the metric $\de$. From smoothness and the boundary conditions, we know that $\ell$ is a smooth invertible metric  for a small time interval $[0,\tau]$ with  $ |\ell-\de|^2 \leq \varepsilon(n)$  during this time interval. By \eqref{eq:RicciDeTurckdelta} we can compute
\begin{equation*}
\begin{split}
\partt |\ell-\de|^2 &\leq 2(\ell-\de)^{ij} \ell^{ab}\partial_a \partial_b (\ell-\de)_{ij}
+c(n) |\ell-\de||D \ell|^2 \cr
& \leq \ell^{ab}\partial_a \partial_b |\ell-\de|^2 - 2(D\ell,D\ell)_{\ell}
+c(n) |\ell-\de||D \ell|^2 \cr
& \leq \ell^{ab}\partial_a \partial_b |\ell-\de|^2  -|D\ell|^2 + c(n) \varepsilon(n)   |D \ell|^2 \cr
& \leq \ell^{ab}\partial_a \partial_b |\ell-\de|^2
\end{split}
\end{equation*}
for all $t\in [0,\tau]$ 
if $\varepsilon(n)$ is sufficiently small. Thus by the maximum principle $|\ell-\de|^2 \leq \varepsilon(n)$ remains true as long as this is true on the boundary. Thus we can take $\tau = T$.

We perform a similar calculation  for $|\ell-\ell_0|^2$. By the above estimate, we can freely use that $\frac{1}{2} \de \leq \ell \leq 2 \de,$ 
for all $t\in [0,T]$. We also use, that $|D\ell_0| + |D^2\ell_0| \leq \epsilon(n)$ 
 due to the assumptions.
 \begin{equation*}
 \begin{split}
\partt |\ell-\ell_0|^2 & \leq  2 (\ell-\ell_0)^{ij} \ell^{ab}\partial_a \partial_b\ell_{ij}
+c(n) |\ell-\ell_0||D\ell|^2 \\
 &= 2 (\ell-\ell_0)^{ij} \ell^{ab}\partial_a \partial_b (\ell - \ell_0)_{ij}
+ 2 (\ell-\ell_0)^{ij}\ell^{ab}\partial_a \partial_b (\ell_0)_{ij} +c(n)\varepsilon(n)  |D \ell|^2  \\
& \leq \ell^{ab}\partial_a\partial_b |\ell-\ell_0|^2 -|D(\ell-\ell_0)|^2 +  \varepsilon(n) +c(n)\varepsilon(n) |D \ell |^2 \\
&\leq \ell^{ab}\partial_a \partial_b |\ell-\ell_0|^2-\frac{1}{2}|D\ell|^2 +c(n)\ep(n)
\end{split}
\end{equation*}
if $\varepsilon(n)$ is small enough.
Hence, 
\begin{equation*}
\partt \big( |\ell-\ell_0|^2  -c(n)\ep(n)t\big) \leq  \ell^{ab}\partial_a \partial_b \big(|\ell-\ell_0|^2-\ep(n)c(n)t\big),
\end{equation*}
and, consequently, 
\begin{equation*} 
|\ell-\ell_0|^2 \leq \beta +\ep(n)c(n)t
\end{equation*}
for $t \leq T$, in view of the fact that $|\ell-\ell_0|^2 \leq \beta \leq \beta + \ep(n)c(n)t $ on $\boundary \B_R(0)\times \{t\}$ for $t \leq T$,   and $|\ell-\ell_0|^2 =0$ for $t =0$ on $\overline{\B_R(0)}.$
\end{proof}

We now consider the problem of constructing solutions to the Dirichlet problem for the $\de$-Ricci-DeTurck flow, with boundary data $h$ given on the parabolic boundary $P$ of $B_R(0) \times (0,T)$. We will assume that the boundary data is given as the restriction of $h \in C^{\infty}( \overline{\B_R(0)} \times [0,T])$ which is $\ep(n)$ close in the $C^0$ norm to  $\de$ on $\overline{\B_R(0)} \times [0,T]$. We now explain how to construct a solution to this Dirichlet problem if the compatibility conditions of the first type are satisfied.
\begin{defn}\label{CompDefn}
Let $h \in C^{\infty}( \overline{\B_R(0)} \times [0,T])$ be a smooth family of Riemannian metrics on $\overline{\B_R(0)}.$  We say $h$ satisfies $\text{Comp}_k$, or {\it $h$ satisfies the compatibility conditions of the $k$-th order}, if
$$\frac{\partial^l}{\partial t^l} h(x,0) = L_l(h(x,0)) \qquad \text{for} \ l =1, \ldots, k$$
for all $x \in \partial \B_R(0),$ where $L_l$ is the differential operator of order $2l$ which one obtains by differentiating \eqref{eq:RicciDeTurckdelta} $l$-times with respect to $t$, and inserting iteratively the already obtained formulas for the $m$-th derivative in time for $m = 1, \ldots, l-1$. For example 
\begin{equation*}
\begin{split}
  L_1(h)_{ij}=\,h^{ab}\partial_a\partial_b h_{ij} +\tfrac12 h^{ab}h^{pq}\big(&\partial_i h_{pa}\partial_jh_{qb}
    +2\partial_ah_{jp}\partial_qh_{ib}\\
    &-2\partial_ah_{jp}
    \partial_bh_{iq}-2\partial_jh_{pa}\partial_bh_{iq}
    -2\partial_ih_{pa}\partial_bh_{jq}\big)\, .
\end{split}
\end{equation*}
and
$$L_2(h)_{ij} = -h^{ak}h^{bm}L_1(h)_{km}\partial_a \partial_b h _{ij} + 
 h^{ab} \partial_a \partial_b L_1(h)_{ij} + \ldots\, .$$
\end{defn}

We are now prepared to derive the following existence result.

\begin{thm}\label{Dirichlet}
Let $h \in  C^{\infty}( \overline{\B_R(0)} \times [0,T]),$ $T\leq 1$  and assume 
 $|h(\cdot,t) -\de| \leq \ep(n)$ for all $t \in [0,T],$ and $|h_0 - \de|_{C^{2,\al}(\B_R(0))} \leq \ep(n)$ and
 $h$ satisfies $\text{Comp}_1$. Then there exists a solution $ \ell \in C^0(\overline {\B_R(0)} \times [0,T] ) \cap H^{2+ \al, 1+ \frac {\al}{2} } (\overline{\B_R(0) }\times [0,T])$ to the $\de$-Ricci-DeTurck flow such that 
$\ell|_{P} =  h|_{P}$ on the parabolic boundary $P$. 
Furthermore, 
\begin{eqnarray}
|\ell(x,s)- h(x,s)|\leq    C(n,R,\al) K\big(s^{-1}, \|h\|_{C^{2,1}(\B_R(0) \times [s,T])}\big) (R-|x|)^{\al} \label{Hoelder}
\end{eqnarray}
 for all $s\in (0,T],$ $x \in B_R(0),$ for any given $\al \in (0,1)$, where $K: \R^+ \times \R^+ \to \R$ is a monotone increasing function with respect to each of its argument.

\end{thm}
\begin{proof}
 The first part of the proof follows closely the proof of the existence result given in \cite[Chapter 3]{Shi}.
Assume for the moment that a solution $\ell$ exists, and set $S:= \ell- h$.
Then $S=0$ on the parabolic boundary, and
the evolution equation for $S$ is: 
\begin{equation*}
\begin{split}
 \partt S(x,t)  &= \Big(\partt \ell  -\partt h\Big)(x,t) \\  
&= \Big(\ell^{ij}\partial_i \partial_j \ell + \ell ^{-1}*  \ell ^{-1} * D \ell * D \ell  -\partt h\Big)(x,t)\\
&= \Big(\ell^{ij}\partial_i\partial_j S + \ell^{ij}\partial_i \partial_j h + \ell ^{-1}*  \ell ^{-1} * D \ell * D \ell  -\partt h\Big)(x,t)\\
&= a^{ij}(S(x,t),x,t) \partial_i \partial_j S(x,t) + b(S(x,t), D S(x,t),x,t)
\end{split}
\end{equation*}
where $a^{kl}(z,x,t) := (h(x,t) + z)^{kl}$ is the inverse of $h(x,t) +z$ (which is well defined as long as $h(x,t) +z$ is invertible) and $b$ is defined similarly
\begin{equation*}
\begin{split}
b(z,p,x,t) &:= (h(x,t) +z)^{-1}* (h(x,t) +z)^{-1} *  (p + D  h(x,t) ) * (p + D h(x,t) ))\\
&\ \ \ -\partt h(x,t) + (h(x,t) + z)^{kl} \partial_k \partial_l h(x,t),
\end{split}
\end{equation*}
where this again is well defined as long as $z + h(x,t)$  is invertible. Note that if $|z|$ is sufficiently small then  $z + h(x,t)$  is invertible.
Since $\ell$ is assumed to be a solution, and $\ell = h$ on $P$, where $|h-\de|_{C^0} \leq \ep(n),$ and $|\ell_0 -\de|_{C^2} = |h_0-\de|_{C^2} \leq \ep(n)$,  we obtain that
$| \ell(\cdot,t) -\ell_0|_{C^0(\B_R(0)) } \leq \ep(n)$ for all $t\in [0,T]$ in view of 
Lemma \ref{C0Lemma}, and hence $|S(t)|_{C^0(\B_R(0)) } \leq \ep(n)$ for all $t\in [0,T]$.

We divide $S$ by a small number $\de(n)>0$, and call it $\ti S$, i.e. $$\ti S := \de^{-1}(n)S,\, $$
where we assume $\ep(n) \ll \de(n),$ for example we choose $\ep(n) = \de^3(n)$.
Hence  $|\ti S(t)|_{C^0(\B_R(0))}$ is still small for all times $t\in[0,T]$, and we  
have $|\ti S| \leq \sqrt{\ep(n)} $.
The evolution of $\ti S$ may be written as
\begin{equation}\label{gen}
 \partt \ti S(x,t) = {\ti a}^{ij}(\ti S (x,t),x,t )\partial_i \partial_j \ti S(x,t) + \ti b(\ti S(x,t), D \ti S(x,t), x,t),
 \end{equation} 
where  ${\ti a}^{ij}(\ti z,x,t) = (h(x,t) + \de \ti z)^{ij} $ is the inverse of $h(x,t) +\de \ti z,$ and
\begin{equation*}
\begin{split}
\ti b(\ti z,\ti p,x,t) &:= \de (h(x,t) + \de \ti z)^{-1}* (h(x,t) + \de \ti z)^{-1} * ( \ti p +  \de^{-1} D h(x,t)  ) * ( \ti p +  \de^{-1} D h(x,t)   )  )\\
&\ \ \ -\frac{1}{\de}\partt h(x,t) + (h(x,t) + \de \ti z)^{kl} \de^{-1} \partial_k \partial_l h(x,t),
\end{split}
\end{equation*} 
In the setting we are examining, we see,  defining  $C_1(h,n) :=  10 |h|_{C^{2,1}(\B_R(0) \times [0,1])} $, that 
 $$|\ti b( \ti z,\ti p, x,t) | \leq  \de(n) ( C_1(h)\de^{-2}(n) + |\ti p|^2 ) $$ for the $\ti z= \ti S(x,t)$ we are considering,  since   $|\ti S|_{C^0} \leq \sqrt{\ep(n)}$. 
 As long as $|\ti S(\cdot,t)|_{C^0} \leq \sqrt{\ep(n)}$ for $t\in [0,1]$, we have  
$2\de_{ij} \geq \ti a^{ij}(\ti S (x,t),x,t ) \geq \frac 1 2 \de_{ij}$ and 
$|\ti b( \ti S(x,t), D \ti S(x,t),x,t) | \leq \de(n) (C_1(h) \de^{-2}(n)  +  |D\ti S(x,t)|^2) $ in this  case.
We write this as 
$$ |\ti b( \ti S(x,t), D \ti S(x,t),x,t) | \leq   Q(|D \ti S(x,t)|, |\ti S(x,t)|)(1+ |D \ti S(x,t)|^2) $$
where $Q$ is a smooth function, with  $Q(|p|,|u|) =  2\de(n) \eta (|p| ) + 2\de(n)( 1- \eta(|p|) (1+100C_1(h)\de^{-2}(n))$ 
where $\eta$ is a smooth cut-off function with $\eta(r) =0$ for all $r \leq 100 \, C_1(h) \de^{-2}(n)$ and $ \eta(r) = 1$
for all $r \geq 200\, C(h,n) \de^{-1}(n)$.

Hence, from the general theory of non-linear parabolic equations of second order, see for example  \cite[Theorem 7.1, Chapter VII]{LSU}, we see that equation \eqref{gen} with zero parabolic boundary values has a solution $\ell \in H^{2+\al,1 + \al/2}(\overline{\B_R(0)}\times [0,T])$
for all times $t\in [0,T]$, as long as $|\ti S(\cdot,t)|_{C^0} \leq \sqrt{\ep(n)}$ for $t\in [0,T]$.

 Writing $\ell =h + \de \ti S$ and using the arguments above, we see that this will not be violated for $t\in [0,T]$, $T\leq 1$, and that $\ell$ solves the $\de$-Ricci-DeTurck equation and $\ell|_P = h|_P$.

This proves the existence of the solution. It remains to prove the H\"older boundary estimate, \eqref{Hoelder}. For ease of reading, we assume $R=1$.




Consider  for $q,r$ fixed the functions 
$$\phi_+:= \big(  (\ell_{qr} -h_{qr}) + \lambda |\ell-h|^2 \big) \ \text{ and }\ \phi_{-}:= \big( - (\ell_{qr} -h_{qr}) + \lambda |\ell-h|^2 \big),$$ where $\lambda=\lambda(n)$ is a sufficiently large constant such that 
$$\partial_t\phi_\pm-\ell^{ij}\partial_i \partial_j \phi_\pm \leq C(n,h|_{[s,T]}), $$ 
on $\B_1(0)  \times [s,T]$ where $C(n,h) =  C(n,\norm{  h }_{ C^{2,1}( \B_1(0) \times [s,T])    })$.

We consider the functions 
$$\psi_+ := \eta(t)\phi_{+} - 2M\rho^{\al} \ \text{ and }\  \psi_{-} := \eta(t)\phi_{-}- 2M\rho^{\al}\, ,$$ for some $0<\al<1$, 
where $\rho(x) = (1-|x|)$, and $\eta$ is a non-negative cut off function in time with $\eta(t) = 0$ for $0\leq t\leq s$ and  $\eta(t) = 1$ for $3s/2 \leq t$ such that $|\eta'(t)|^2\leq cs^{-2}\eta(t)$ for some positive constant $c$.
A direct calculation shows 
\begin{equation*}
\begin{split}
\ell^{ij}\partial_i \partial_j \rho^{\al} & = \ell^{ij} \partial_i \left(\al \rho^{\al-1} \frac{(-x_j)}{|x|}\right) \cr
& =  (\al-1)\al\rho^{\al-2}  \ell^{ij} \frac{x_i x_j }{|x|^2} - \al \rho^{\al-1} \frac{\ell^{ii}}{|x|} +  \al \rho^{\al-1}\frac{\ell^{ij}x_i x_j}{|x|^3}  \cr
&\leq - \frac{\al(n-1)}{2}   \frac{\rho^{\al-1}}{|x|}- \frac{\al |1-\al|}{2} \rho^{\al-2} \\
&\leq -1 - \frac{\al |1-\al|}{2} \rho^{\al-2}
\end{split}
\end{equation*} 
for all $|x|\in (1- \de_0(\al,n),1)$.
We note that $\psi_\pm$ cannot be zero very close to  $\boundary \B_R(0)$, since
$|D (\ell-h)|$  is bounded by {\it some} constant according to \cite[Theorem 7.1, Chapter VII]{LSU}, and $\ell-h =0$ on $\boundary \B_R(0)$.
Also, by choosing $M=M(\al)$ large enough, we have without loss of generality, that
 $\psi_\pm(x,\cdot) <0$ for $|x| = 1-\de_0(\al,n).$
 That is $\psi_\pm(x,\cdot) <0$ for $|x| =1-\de_0(\al,n)$ and $|x|=  1- \ep$
for all $\ep>0$, $\ep\ll \de_0$  sufficiently small.
We also observe that    $\psi_\pm(\cdot,s) < 0$ for all $|x| \in [1-\de_0(\al,n),1-\ep]$
 for  $t\leq s$. Hence   if $\psi_\pm(x,t) = 0$ for some $|x| \in [1-\de_0(\al,n),1-\ep]$ for some $t\geq s$, there must be a first time $t$ for which this happens and 
this must happen at an interior point $ x $ of $\B_{1-\ep}(0) \backslash \overline{\B_{1-\de_0}(0)}$. 
We calculate at such a point $(x,t)$,
\begin{equation*}
\begin{split}
\partt \psi_\pm  & \leq  \ell^{ij}\partial_i \partial_j \psi_\pm  +C(n,h) + \eta' \phi_\pm   -2M - M \al |1-\al| \rho^{\al-2}  \cr
&  \leq  \ell^{ij}\partial_i \partial_j  \psi_\pm  + C(n,h) + \frac{c(n)}{s}|\eta\phi_\pm |^{\frac 1 2}   -2M  - M \al |1-\al| \rho^{\al-2} \cr
&\leq  \ell^{ij}\partial_i \partial_j  \psi_\pm  + C(n,h) -2M+ \frac{c(n)}{s}(2M)^{\frac 1 2} \rho^{\frac {\al}{2}}-  M \al |1-\al| \rho^{\al-2}  <0\, ,
\end{split}
\end{equation*}
for $|x| \in (1-\de_0,1-\ep),$ if $M> C( C(n,h),\al,s) $ also holds.
A contradiction.
This leads to the desired  estimate close to the boundary. For points $|x|\in (0, 1-\de_0)$ the estimate follows immediately from the fact that $\ell$ and $h $ are $\ep(n)$ close to $\de$ and hence bounded, and $(1-|x|)^{\al} \geq \de_0^{\al}$  for $|x| \in  (0, 1-\de_0)
.$ \end{proof}

If we assume that higher order compatibility conditions are satisfied, we obtain more regularity of the solution.
\begin{thm}\label{Dirichlet2}
Let $h \in  C^{\infty}( \overline{\B_R(0)} \times [0,T]),$ $T\leq 1$ and assume 
 $|h(\cdot,t) -\de| \leq \ep(n)$ for all $t \in [0,T],$ and $|h_0 - \de|_{C^{2,\al}(\B_R(0))} \leq \ep(n)$ and
 $h$ satisfies $\text{Comp}_k$. 
Then there exists a solution $ \ell \in C^0(\overline {\B_R(0)} \times [0,T] ) \cap H^{k+ \al, \frac{k}{2}+ \frac {\al}{2} } (\overline{\B_R(0)} \times [0,T])$ to the $\de$-Ricci-DeTurck flow and 
the values given by $h$ on the parabolic boundary, that is 
$\ell|_{P} =  h|_{P}$. 
Furthermore, 
\begin{equation}
|\ell(x,s)- h(x,s)|\leq    C(n,R,\al) K\big(s^{-1}, \|h\|_{C^{2,1}(\B_R(0) \times [s,T])}\big) (R-|x|)^{\al} 
\end{equation}
 for all $s\in (0,T],$ $x \in B_R(0),$ for any given $\al \in (0,1)$, where $K: \R^+ \times \R^+ \to \R$ is a monotone increasing function with respect to each of its argument. 
\end{thm}
\begin{proof}
The proof is the same, except at the step where we used \cite[Theorem 7.1, Chapter VII]{LSU} to obtain a solution in $H^{2+\al,1 + \frac{\al}{2}},$  we now obtain a solution $\ell \in 
H^{k+\al,k + \frac{\al}{2}},$ in view of the fact that the $\ti S$ satisfies the compatibility condition of $k$-th order.
\end{proof}

We now explain how to construct a $\de$-Ricci-DeTurck  flow for  parabolic boundary values given by $h$ which do not necessarily satisfy compatibility conditions of the first order, but are smooth at $t=0$, 
smooth on $\overline{\B_R(0) }\times (0,T]$ and continuous on $\overline{\B_R(0)} \times [0,T]$. This is done by modifying the boundary values, so that the first (or higher  order) compatibility conditions are satisfied, and then taking a limit.

\begin{thm}\label{Dirichlet3}
Let $h \in  C^{\infty}(\overline{\B_R(0)} \times (0,T]) \,\cap\, C^0( \overline{\B_R(0)} \times [0,T])$, $T\leq 1$,  such that $h(\cdot,0) \in C^{\infty}(\overline{\B_R(0)})$ and assume 
 $|h(\cdot,t) -\de| \leq \ep(n)$ for all $t \in [0,T],$ and $|h_0 - \de|_{C^{2,\al}(\B_R(0))} \leq \ep(n)$. 
Then there exists a solution 
$$ \ell \in C^0(\overline {\B_R(0)} \times [0,T] ) \cap C^{\infty}(\B_R(0) \times (0,T])$$  
to the $\de$-Ricci-DeTurck flow and 
the values given by $h$ on the parabolic boundary $P$, that is 
$$\ell|_{P} =  h|_{P}\, .$$ 
Furthermore 
$$|\ell|_{C^s(\B_{R'}(0) \times [0,T])} \leq c\big(|R-R'|, |h_0|_{C^{s}(\B_{R'}(0))}\big)$$ 
for any $R' <R$ and any $s  \in \N$.
\end{thm}

\begin{proof}

Let $\xi: \R \rightarrow \R$ be a monotone non-increasing smooth function whose image is contained in $[0,1]$, such that $\xi$ is equal to $1$ on $[0,\frac 1 2]$ and equal to $0$ on $[1, \infty)$.

For each $\tau \in [0,1]$, let $h(\tau)$  be the smooth Riemannian metric defined as follows:
\begin{equation*}
\begin{split}
h(\tau)|_{ \B_R(0)  \times [\tau,T]} &:= h ,\\
 h(\tau)|_{ \B_R(0)  \times [0,\tau]} &:=  \xi\left(\frac{t}{\tau}\right)h_0(x)+ \left(1- \xi\left(\frac{t}{\tau}\right)\right)h(x,t) +  t\xi\left(\frac{t}{\tau}\right) L_1(h_0)(x),
 \end{split}
 \end{equation*}
   where  
$L_1$ is as in Definition \ref{CompDefn}, i.e.
\begin{equation*}
L_1(h_0) :=(h_0)^{ij} \partial_i \partial_j h_0+(h_0)^{-1} * (h_0)^{-1} * D h_0 * D h_0.
\end{equation*}
Using $|h(t)-\de|\leq \ep(n)$ and $|L_1(h_0)| \leq c(n)$, we see that
$ |h(\tau)(t)-\de| \leq 2\varepsilon(n) = \varepsilon(n)$ if $\tau$ is sufficiently small. 
Furthermore,  
$$ \partt h(\tau)(x,0) = L_1(h_0)(x,0)  = L_1(h(\tau))(x,0),$$ that is $h(\tau)$ satisfies $\text{Comp}_1$.
We may hence use Theorem \ref{Dirichlet} to obtain a solution $\ell(\tau) \in H^{2+ \al, 1 +\frac{\al}{2}}(\overline{\B_R(0)} \times [0,T] )$ to the $\delta$-Ricci-DeTurck flow with parabolic boundary data given by $h(\tau)$.
From the definition of $h(\tau)$, we have on $\boundary \B_R(0)$ that 
\begin{equation*}
\begin{split}
|\ell(\tau)(\cdot,t) - h_0)|  &= |h(\tau)(\cdot,t)- h_0(\cdot)|\cr
&  \leq |h(\cdot,t)-h_0(\cdot)| + t \ep(n)\cr
& \leq C(t,h,n)
\end{split}
\end{equation*}
where $C(t,h,n)$ is a function (independent of $\tau$) such that $C(t,h,n) \leq \varepsilon(n)$ and
$C(t,h,n) \to 0$ as $t\downto 0$ for $n$ and $h$  fixed.
 Lemma \ref{C0Lemma} then tells us that 
$|\ell(\tau)(\cdot,t)  - h_0)| \leq \hat C(t,h,n)$ on {\it all} of $\overline{\B_R(0)}$ for all 
$t \in [0,T]$  where  $\hat C(t,h,n) \to 0$ as $t\downto 0$ for $n$  and $h$ fixed (independent of $\tau$).
The boundary H\"older  estimate \eqref{Hoelder} of Theorem \ref{Dirichlet}, and 
the smoothness of $h = h(\tau)$ for $t \geq 2\tau$, also tells us, for any $\ep>0$ and $s>0$, there exists a $\si>0$, such that 
$|\ell(\tau)(x,t)  - h(x,t)| \leq \ep$ for all $x \in \B_R(0) \backslash \B_{R-\si}(0) $ for  all $t\in [s,T]$, where $\si= \si(\ep,s,h,n) >0$ is  independent of $\tau$ if $\tau $ is sufficiently small.
These two $C^0$ estimates imply that we have the   
uniform (in $\tau$) $C^0$ bounds
$$|\ell(\tau)(t) - h(t)|_{P_{\ep}} \leq C(\ep,h,n)$$ 
where $C(\ep,h,n) \to 0$ as $\ep \to 0$ (for fixed $h$ and $n$ ) and $P_{\ep} = \overline{\B_R(0)} \backslash \B_{R-\ep}(0) \times [0,1] \cup \overline{\B_R(0)} \times [0,\ep]$ for $\tau$ sufficiently small.

If we define 
$$h(\tau) := \xi(t/\tau)h_0(x)+ \left(1- \xi(t/\tau)\right)h(x,t) +  t\xi(t/\tau) L_1(h_0)(x) 
+\frac{t^2}{2} L_2(h_0),$$ 
then we still have 
$|\ell(\tau)(\cdot,t) - h_0)| \leq C(t,h,n)$ and hence 
$|\ell(\tau)(\cdot,t)  - h_0)| \leq \hat C(t,h,n)$ on {\it all} of $\overline{\B_R(0)}$ for all 
$t \in [0,T]$  where  $\hat C(t,h,n) \to 0$  independent of $\tau$ if $\tau $ is sufficiently small, in view of Lemma \ref{C0Lemma}. 

The H\"older  estimates still hold:  for any $\ep>0$ and $s>0$, 
$|\ell(\tau)(x,t)  - h(x,t)| \leq \ep$ for all $x \in \B_R(0) \backslash \B_{R-\si}(0) $ for all $t\in [s,T]$, where $\si= \si(\ep,s,h,n) >0$ is independent of $\tau$ if $\tau $ is sufficiently small.
Thus,  the uniform $C^0$ estimates  
$|\ell(\tau)(t) - h(t)|_{P_{\ep}} \leq C(\ep,h,n)$ where $C(\ep,h,n) \to 0$ as $\ep \to 0$ (for fixed $h$ and $n$) 
still hold.
Continuing in this way, we can assume $\ell(\tau) \in  H^{k+ \al, \frac{k}{2} +\frac{\al}{2}}(\overline{\B_R(0)} \times [0,T] )$
and $|\ell(\tau)(\cdot)- \de|_{C^0(\B_{R}(0) \times [0,T])} \leq \ep(n)$ and $\ell(\tau)(\cdot,0) = h_0$, and the uniform $C^0$ estimates hold,
$|\ell(\tau)(t) - h(t)|_{P_{\ep}} \leq C(\ep)$ where $C(\ep) \to 0$ as $\ep \to 0$ and $P_{\ep} = \overline{\B_R(0)} \backslash \B_{R-\ep}(0) \times [0,T] \cup \overline{\B_R(0)} \times [0,\ep]$.
The proof of the interior estimates, explained in \cite[Section 4]{MilesC0paper} can be used here to show that 
$$|\ell(\tau)(\cdot,t)|_{C^{s}(\B_{R'}(0) )} \leq c\big(|R'-R|, |h_0|_{C^{s}(\B_{R}(0)) }\big)$$ for all $t\in  [0,T]$, for all $R'<R$.
By Arzel\`a-Ascoli's Theorem, one is able to take a limit: up to a subsequence, we obtain a solution $\ell$, with the desired properties. \end{proof}

\section{An $L^2$ estimate for the Ricci-DeTurck flow, and applications thereof}\label{L2Lemma_Section}

In this chapter we prove a  lemma which estimates the change in the $L^2$ distance between two solutions to the
$\de$-Ricci-DeTurck flow.
Lemma \ref{L2Lemma} considers smooth solutions 
which are $\ep(n)$ close in the $L^2$ norm at time zero and agree at all times on the boundary. If we weight the $L^2$ distance at time $t$ of two smooth solutions  appropriately, then this quantity is non-increasing in time. The weight has the property that  
it is uniformly bounded between $1$ and $2$, and hence the unweighted 
 $L^2$ distance at time $t$ of the two  solutions can only increase by a factor of  at most $2$. With the help of the $L^2$-Lemma, we prove some uniqueness theorems for solutions to the $\de$-Ricci-DeTurck flow.
 
 \begin{lemma}[$L^2$-Lemma]\label{L2Lemma}  
 Let $g_{1},g_2$ be two smooth solutions to the $\de$-Ricci-DeTurck flow on  $\B_R(0)  \times (S,T)$ such that $g_l \in C^\infty(\overline{\B_R(0)}\times (S,T))$
 for $l =1,2$ and that $g_1 = g_2$ on $\partial \B_R(0) \times (S,T)$. 
 Let $h:= g_1- g_2$ and 
 $$ v:= |h|^2 \left(1+ \lambda ( |g_1-\delta|^2 + |g_2-\delta|^2)\right)\, .$$ 
 We assume that $|g_1(\cdot,t) -\de| + |g_2(\cdot,t) -\de|  \leq \ep(n)$ 
for all $t\in (S,T)$. Then for $\lambda \geq \hat \lambda(n)$ and $\varepsilon \leq \hat \ep(n)$, where $\hat \ep(n)$ is  sufficiently small and $\hat \lambda (n)$ sufficiently large, 
 it holds for $t\in (S,T)$ that
$$ \partt \int_{\B_R(0)}v\, dx \leq 0\, .$$
\end{lemma}
Before proving the lemma, we state and prove two corollaries of this estimate.
\begin{cor}\label{CorL2Lemma}
Let $g_1,g_2$ be two smooth solutions to the $\de$-Ricci-DeTurck flow on  $\B_R(0)  \times (0,T)$ such that $g_l \in C^\infty(\overline{\B_R(0)}\times (0,T)) \cap C^0(\overline{\B_R(0) }\times [0,T))$ for $l =1,2$ and that $g_1 = g_2$ on $\partial \B_R(0) \times [0,T)$ and $g_1(\cdot,0) = g_2(\cdot,0)$. 
We assume further that $|g_1(\cdot,t) -\de| + |g_2(\cdot,t) -\de|  \leq \ep(n)$ 
for all $t\in (0,T)$.  
Then $g_1 =g_2$ on $\B_R(0) \times [0,T).$
\end{cor}

\smallskip

\begin{rmk}

\cite[Proposition 7.51]{CLN} shows a uniqueness result of the Ricci-DeTurck flow with a background metric $g$ with bounded curvature for solutions $(g(t))_{t\in(0,1)}$ that behave as follows: there exists a positive constant $A$ such that $A^{-1}g\leq g(t)\leq Ag$ and $|\nabla^gg(t)|+\sqrt{t}|\nabla^{g,2}g(t)|\leq A$ for all $t\in(0,1)$. Its proof is based on the maximum principle. Corollary \ref{CorL2Lemma} assumes a stronger condition on the closeness to the background Euclidean metric but it does not assume any a priori bounds on the first and second covariant derivatives: the  proof is based on energy estimates.

\end{rmk}
\begin{proof}[Proof of Corollary \ref{CorL2Lemma}] 
Let $h := g_{1}-g_{2}$ and  $$ v:= |h|^2 \left(1+ \lambda ( |g_{1}-\delta|^2 + |g_{2}-\delta|^2)\right)\, .$$
From the assumptions, we know that 
$ |v(\cdot,0)|  =0$ 
 on  $\overline{\B_R(0)}$ and hence 
 \begin{equation*}
 \int_{\B_R(0)} v(x,\tau)\, dx \leq \si(\tau), 
 \end{equation*}
  where $\si(\tau)$ tends to $0$ with $\tau \downto 0,$ in view of the continuity of $v$.
We also have $ g_1(\cdot,t) = g_2(\cdot,t)$ on $\boundary \B_R(0)$ for all $t \in [0,T]$, and so
 Lemma \ref{L2Lemma} implies $ \int_{\B_R(0)} v(x,t) \,dx \leq \si(\tau)$ for all $t\in (\tau,T)$.
Taking a limit $\tau \downto 0$, we see $ \int_{\B_R(0)} v(x,t)\, dx  =0$ for  all 
$t\in (0,T)$.
 This implies $g_1(\cdot,t) = g_2(\cdot,t)$ for all $t\in [0,T)$ as required.
\end{proof}
By slightly modifying the previous proof, we can also show the following uniqueness statement.
\begin{cor}\label{Cor2L2Lemma}
Let $h$ be a smooth solution to the $\de$-Ricci-DeTurck flow  on  $\B_R(0)  \times (0,T)$ such that $h\in C^\infty(\overline{\B_R(0) }\times (0,T)) \cap C^0(\overline{\B_R(0) }\times [0,T))$  and assume $h_0 \in C^{\infty}(\B_R(0))$ and let $\ell$ be the solution constructed in Theorem \ref{Dirichlet3}, with parabolic boundary data defined by $h|_P$. We assume further that $|h(\cdot,t) -\de| \leq \ep(n)$ for all $t\in [0,T)$. Then $\ell =h$.
\end{cor}
\begin{proof}
Let $h(\tau)$ be the modified metric defined in the proof of Theorem \ref{Dirichlet3}, and $\ell(\tau)$  the solutions defined there.
Let 
\begin{equation*}
 v(\tau) := |h(\tau) - \ell(\tau)|^2 \left(1+ \lambda ( |h(\tau)-\delta|^2 + |\ell(\tau)-\delta|^2)\right).
 \end{equation*}
From the construction of $h(\tau)$ we know that 
$ |v(\tau)(\cdot,\tau)|  \leq \si(\tau)$ on  $\B_R(0),$ where $\si(\tau) \to 0$ with $\tau \downto 0.$ This implies 
$$ \int_{\B_R(0)} v(\tau)(x,\tau)\, dx \leq \si(\tau), $$ where $\si(\tau)\to 0$ with $\tau \downto 0.$ 
Since  $h = h(\tau) = \ell(\tau)$ on $\boundary \B_R(0)$ for all $t \in [\tau,T]$ and the solution $\ell(\tau)$ is $C^k$ (in space and time) up to and including the boundary, we can use 
Lemma \ref{L2Lemma} to conclude $ \int_{\B_R(0)} v(\tau)(x,t) \,dx \leq \si(\tau)$ for all $t\in (\tau,T)$.
Taking a limit $\tau \downto 0$, we see $ \int_{\B_R(0)} v(x,t) dx \leq 0$ for  all 
$t\in (0,T)$, with $ v := |h-\ell|^2 \left(1+ \lambda ( |h-\delta|^2 + |\ell-\delta|^2)\right)\, $. This implies $h(\cdot,t) = \ell(\cdot,t)$ for all $t\in [0,T)$ as required.
\end{proof}

\begin{proof}[Proof of Lemma \ref{L2Lemma}]

In the following, we will assume that $\ep(n)$ is a small positive constant, and $\la(n) = 1/\sqrt{\ep(n)}  $ is a large constant, which  satisfies
$\lambda(n) \ep(n) = \sqrt{\ep(n)} =: \si(n)$ per definition.

From \eqref{eq:RicciDeTurckdelta} we have for $g_{l},$ $l\in \{1,2\}$:
$$ \partt g_l = g_l^{ab} \partial_a \partial_b g_l + g_l^{-1} * g_l^{-1}*D g_l *D g_l,$$
and
$$ \partt |g_l-\delta|^2 \leq g_l^{ab} \partial_a \partial_b |g_l-\delta|^2 - \frac{2}{1+\varepsilon} |D g_l|^2 .$$
Summing over $l=1,2$, and writing $\tilde{h}^{ab}:= \frac 12 \left(g_1^{ab}+g_2^{ab}\right)$ and $\hat{h}^{ab}:= \frac 12 \left(g_1^{ab}-g_2^{ab}\right)$, we get 
\begin{align*}
\partt ( |g_1-\, &\de|^2 + |g_2 - \de|^2 ) \\
&  \leq g_1^{ab} \partial_a \partial_b |g_1-\delta|^2 + g_2^{ab} \partial_a \partial_b |g_2-\delta|^2 - 2(1-\ep(n)) |D g_1|^2 - 2(1-\ep(n)) |D g_2 |^2\\
&  = {\ti h}^{ab}  \partial_a \partial_b ( |g_1-\delta|^2 + |g_2-\delta|^2)  + \hat  h^ {ab} \partial_a \partial_b ( |g_1-\delta|^2 - |g_2-\delta|^2)\\
& \ \ \ - 2(1-\ep(n))|D g_1|^2- 2(1-\ep(n))|D g_2 |^2\\
& = {\ti h}^{ab}  \partial_a \partial_b ( |g_1-\delta|^2 + |g_2-\delta|^2)  + \sum_{l=1}^2(\hat h *  (g_{l} -\de)* D^2 g_l + 
\hat h *  D  g_l * D g_l)\\
& \ \ \ - 2(1-\ep(n))|D g_1|^2- 2(1-\ep(n)) |D g_2 |^2\\
& \leq {\ti h}^{ab}  \partial_a \partial_b ( |g_1-\delta|^2 + |g_2-\delta|^2)
+ \sum_{l=1}^2(\hat h *  (g_{l} -\de)* D^2 g_l)  \\
&\ \  \ - 2(1-2\ep(n)) |D g_1|^2- 2(1-2\ep(n)) |D g_2 |^2,
\end{align*}
in view of the fact that $g_l$ is $\ep$ close to $\de$ for $l=1,2$.

We obtain for the difference $h = g_1-g_2$ that 
\begin{equation*}
\begin{split}
\partt h &=  g_1^{ab} \partial_a \partial_b g_1 + g_1^{-1} * g_1^{-1}*D g_1 *D g_1\\
&\ \ \ - g_2^{ab} \partial_a \partial_b g_2- g_2^{-1} * g_2^{-1}*D g_2 *D g_2 \\
&= \frac 12 \left(g_1^{ab}+g_2^{ab}\right)\partial_a \partial_b h + \frac 12 \left(g_1^{ab}-g_2^{ab}\right)\partial_a \partial_b (g_1+g_2)\\
&\ \ \ + (g_1^{-1} - g_2^{-1})* g_1^{-1} * D g_1*D g_1 + g_2^{-1}* (g_1^{-1}-g_2^{-1}) * D g_1*D g_1 \\
&\ \ \ + g_2^{-1}* g_2^{-1} * D h *D g_1 +  g_2^{-1}* g_2^{-1} * D g_2 *D h,
\end{split}
\end{equation*}
which we can write as 
\begin{equation*}
\begin{split}
\partt h &= \tilde{h}^{ab}\partial_a \partial_b h + \hat{h}^{ab}\partial_a \partial_b (g_1+g_2)\\
&\ \ \ + \hat{h}* g_1^{-1} * D g_1*D g_1 + g_2^{-1}* \hat{h}* D g_1*D g_1 \\
&\ \ \ + g_2^{-1}* g_2^{-1} * D h *D g_1 +  g_2^{-1}* g_2^{-1} * D g_2 *D h.
\end{split}
\end{equation*}
This implies that
\begin{equation*}
\begin{split}
\partt |h|^2 &\leq \tilde{h}^{ab}\partial_a \partial_b |h|^2 - \frac{2}{1+\ep} |D h|^2 + h * \hat{h} *D^2 (g_1+g_2)\\
&\ \ \ + h*\hat{h}* g_1^{-1} * D g_1*D g_1 + h*g_2^{-1}* \hat{h}* D g_1*D g_1 \\
&\ \ \ + h*g_2^{-1}* g_2^{-1} * D h *D g_1 + h* g_2^{-1}* g_2^{-1} * D g_2 *D h\,,
\end{split}
\end{equation*}
in view of the fact that  $\ti h$ is $\ep$ close to $\de$.
We now consider the test-function
$$ v:= |h|^2 \left(1+ \lambda ( |g_1-\delta|^2 + |g_2-\delta|^2)\right).\, $$
 We obtain
\begin{equation*}
\begin{split}
\partt v &\leq  \tilde{h}^{ab}\partial_a \partial_b v -2 \lambda \tilde{h}^{ab} \partial_a |h|^2 \partial_b ( |g_1-\delta|^2 + |g_2-\delta|^2)\\
&\ \ \ -  2\lambda(1- 2\varepsilon(n)) |h|^2 \left(|D g_1|^2 + |D g_2|^2\right)\\
&\ \ \   - 2(1- 2\ep(n)) \left(1+ \lambda ( |g_1-\delta|^2 + |g_2-\delta|^2)\right) |D h|^2\\
 &\ \ \ + \left(1+ \lambda ( |g_1-\delta|^2 + |g_2-\delta|^2)\right) \Big( h * \hat{h} *D^2 (g_1+g_2)
 + h*\hat{h}* g_1^{-1} * D g_1*D g_1\\
 &\ \ \ + h*g_2^{-1}* \hat{h}* D g_1*D g_1  + h*g_2^{-1}* g_2^{-1} * D h *D g_1 + h* g_2^{-1}* g_2^{-1} * D g_2 *D h\Big)
 \\[-1.5ex]
 & \ \ \ + \la |h|^2 \sum_{l=1}^2(\hat h *  (g_{l} -\de)* D^2 g_l).
\end{split}
\end{equation*}
Using Young's inequality and the fact that $g_l$ is $\ep$ close to $\de$, for $l=1,2$, as well as  $|h*\hat h|  \leq c(n)|h|^2$, we see that the first order terms appearing in the large brackets may be absorbed by the two negative first order gradient  terms which appear just  before the large brackets, if $\lambda  \geq \Lambda(n)$, where $\Lambda(n)$  is sufficiently large. That is, we have
\begin{equation}
\begin{split}
\partt v &\leq  \tilde{h}^{ab}\partial_a \partial_b v - {2} \lambda \tilde{h}^{ab} \partial_a |h|^2 \partial_b ( |g_1-\delta|^2 + |g_2-\delta|^2)\\
&\ \ \ - \frac{3}{2}\lambda |h|^2 \left(|D g_1|^2 + |D g_2|^2\right)  - \frac 3 2 |D h|^2\\
 &\ \ \ + \left(1+ \lambda ( |g_1-\delta|^2 + |g_2-\delta|^2)\right) \Big( h * \hat{h} *D^2 (g_1+g_2)\Big)\, \\[-1ex] 
&  \ \ \ + \la |h|^2 \sum_{l=1}^2(\hat h *  (g_{l} -\de)* D^2 g_l).\label{evo-equ-v}
\end{split}
\end{equation}
The second term on the right-hand side of (\ref{evo-equ-v}) can be estimated as follows:
\begin{equation*}
\begin{split}
 \left|2 \lambda \tilde{h}^{ab} \partial_a |h|^2 \partial_b ( |g_1-\delta|^2 + |g_2-\delta|^2) \right|&\leq  \la \Big|\ti h *D h * h *  ( |g_1-\de| + |g_2-\de|)\\
 & \qquad\qquad \qquad\qquad\qquad\qquad* (D g_1 + D g_2) \Big|\\
 &\leq c(n) \ep(n) \la  |h||D h| (|D g_1| + |D g_2|).
 \end{split}
 \end{equation*}
Therefore,  it can also be absorbed by the negative terms just before the big brackets, in view of the fact that $\ep(n) \la  \leq \sqrt{\ep(n)}$. This leads to
\begin{equation*}
\begin{split}
\partt v &\leq  \tilde{h}^{ab}\partial_a \partial_b v -   \lambda |h|^2 \left(|D g_1|^2 + |D g_2|^2\right)  -   |D h|^2\\
 &\ \ \ + \left(1+ \lambda ( |g_1-\delta|^2 + |g_2-\delta|^2)\right) \Big( h * \hat{h} *D^2 (g_1+g_2)\Big)\,  \\[-1ex] 
 & \ \ \ + \la |h|^2 \sum_{l=1}^2\big(\hat h *  (g_{l} -\de)* D^2 g_l\big).
\end{split}
\end{equation*}
In order to estimate the second order terms appearing in the equation (\ref{evo-equ-v}), we integrate over $\B:=\B_R(0)$:
\begin{equation*}
\begin{split}
\partt \int_\B v &\leq   \int_{\B}\tilde{h}^{ab}\partial_a \partial_b v   - \lambda \int_{\B} |h|^2 \left(|D g_1|^2 + |D g_2|^2\right)  -  \int_{\B} |D h|^2\\
& \ \ \ + \int_{\B} \left(1+ \lambda ( |g_1-\delta|^2 + |g_2-\delta|^2)\right) ( h * \hat{h} *D^2 (g_1+g_2)) \\
& \ \ \ + \int_{\B}  \la \sum_{l=1}^2|h|^2 (\hat h *  (g_{l} -\de)* D^2 g_l).
\end{split}
\end{equation*}
Since $D v$ and $h$ are zero on the boundary of $\B$, we obtain no boundary terms when integrating the first and last two terms on the right hand side of the above by parts. Doing so, we get
\begin{equation*}
\begin{split}
\partt \int_{\B} v &\leq   -\int_{\B}\partial_a \tilde{h}^{ab}\partial_b v   - \lambda \int_{\B} |h|^2 \left(|D g_1|^2 + |D g_2|^2\right) - \int_{\B} |D h|^2\\
& \ \ \ + \int_{\B} D \Big( \big(1+ \lambda ( |g_1-\delta|^2 + |g_2-\delta|^2)\big) * h * \hat{h} \Big )    *  D (g_1+g_2) \\
& \ \ \ + \int_{\B} \sum_{l=1}^2\la 
D  \Big( |h|^2( \hat h \ast  (g_{l} -\de)) \Big) * D g_l\\
& =: A+ B +C +D + E\, . 
\end{split}
\end{equation*}
 Note also that $|\hat h| \leq c(n) |h|$. 
We estimate the integrand of $A$ as follows:
\begin{equation*}
\begin{split}
 |\partial_a (\ti h)^{ab}  \partial_b v| & \leq c(n) ( |D g_1| + |D g_2|)|D v| \\
&  = c(n)( |D g_1| + |D g_2|) \left|D (|h|^2(1+ \la ( |g_1-\de|^2 + 
|g_2- \de|^2)))\right|\\
& \leq c(n)( |D g_1| + |D g_2|) \Big(  2|h||D h| 
+ |h|^2 \la \ep(n) ( |D g_1| + |D g_2|) \Big) \\
& \leq c(n)( |D g_1| + |Dg_2|)|h||D h|
+ |h|^2 \la \ep(n) ( |D g_1| + |D g_2|)^2,
\end{split}
\end{equation*}
and hence the integral $A$ can be absorbed by the integrals $B$ and $C$.
In estimating the integral of $D,$ we will use  
$$|D \hat{h}| \leq c(n) (|D h| + |h| |D g_1| + |h| |D g_2|),$$ 
the validity of which can be seen by writing,
$\hat h = \frac{1}{2}( (g_1)^{-1} - (g_2)^{-1})  = 
\frac{1}{2} g_1^{-1}(g_2 -g_1)g_2^{-1} = -\frac{1}{2} g_1^{-1} h g_2^{-1}$,  differentiating, and keeping in mind that $g_1$ and $g_2$ are $\ep$ close to $\de$.

We estimate the integrand of $D$ as follows:
\begin{equation*}
\begin{split}
 D \Big( \big(1+ \lambda ( |g_1-\delta|^2 + |g_2-\delta|^2&)\big) * h * \hat{h} \Big )    *  D (g_1+g_2)  \\
& = D \Big( \big(1+ \lambda ( |g_1-\delta|^2 + |g_2-\delta|^2)\big)  \Big)
* h * \hat h *  D (g_1+g_2) \\
& \ \ \ + 
\Big( \big(1+ \lambda ( |g_1-\delta|^2 + |g_2-\delta|^2)\big)  \Big)
D (h * \hat h) *  D (g_1+g_2) \\
& \ \ \ \leq c(n) \ep(n) \la |h|^2 (|D g_1|^2 + |D g_2|^2)\\
& \ \ \ + c(n) \Big(|Dh||h| + |h|^2(|D g_1 | + |D g_2|)  \Big) ( |D g_1| + |D g_2|),
\end{split}
\end{equation*}
and hence the integral $D$ can also be absorbed by the integrals $B$ and $C$.
We estimate the final integral of $E$ in a similar way:
the integrand of $E$ can be estimated by
\begin{equation*}
\begin{split}
 \Big|\sum_{l=1}^2\la 
D  \big( |h|^2 (\hat h *  (g_{l} -\de) )\big) * D g_l\Big| 
& \leq \la c(n)  |h|^2|D h|\ep(n) ( |D g_1| + |D g_2|) \\[-2ex]
&\ \ \ + \la c(n) |h|^3  ( |D g_1|^2 + |D g_2|^2)
\end{split}
\end{equation*}
and hence the integral $E$ can also be absorbed by the integrals $B$ and $C$, in view of the fact that $|h| \leq \ep(n)$.
The result is 
$$\partt \int_{\B} v \leq 0\, ,$$ 
as required.
\end{proof}

\section{Smoothness of solutions coming out of smooth metric spaces}\label{application}

Let $(M,g(t))_{t\in (0,T)}$ be a smooth solution to Ricci flow satisfying
\begin{equation}
\begin{split}
|\Rm(\cdot,t)| &\leq \frac{c_0^2}{t} \\ 
\Rc(g(t)) &\geq -1 \label{standard_est}
\end{split}
\end{equation}
for all $t\in (0,T)$.

We first give an example for an application of Theorem \ref{smoothness_of_solutions_intro} in the setting of expanding gradient Ricci solitons. As explained in the introduction, 
expanding gradient Ricci solitons coming out of smooth cones $(\R^n,d_X,o) = (\R^+ \times S^{n-1} ,dr^2 \oplus r^2 \ga,o) $ where $\ga$ is a Riemannian metric on the sphere, which is smooth and whose curvature operator has eigenvalues larger than one, are examples  of solutions  which satisfy the estimates above. 
In \cite{SchulzeSimon} examples are constructed, and \cite{DeruelleExpanders} it is shown that there is always a solution which comes out smoothly, in the sense that the convergence is in the $C^{\infty}_{loc}$ sense away from the tip.
The uniqueness of such solitons is unknown. Below, we make precise the meaning of an expanding gradient Ricci soliton which comes out of a metric cone.

Recall that an expanding gradient Ricci soliton is a triple $(M^n,g,\nabla^gf)$ where $M$ is a $n$-dimensional Riemannian manifold with a complete Riemannian metric $g$ and a smooth potential function $f:M\rightarrow \R$ satisfying the equation 
\begin{eqnarray}
\Rc(g)-\Hess f=-\frac{g}{2}.\label{stat-sol-equ}
\end{eqnarray}
Also, to each expanding gradient Ricci soliton, one may associate a self-similar solution of the Ricci flow. Indeed, let $(\psi_t)_{t>0}$ be the flow generated by $-\nabla^g f/t$ such that $\psi_{t=1}=Id_M$ and define $g(t):=t\psi_t^*g$ for $t>0$. Then $(M,g(t))_{t>0}$ defines a Ricci flow thanks to (\ref{stat-sol-equ}).
Next, we notice that if an expanding gradient Ricci soliton $(M^n,g(t),p)_{t>0}$ admits a limit as $t$ tends to $0$ in the pointed Gromov-Hausdorff sense for some point $p$ that lies in the critical set of the potential function $f$ then this limit must be the asymptotic cone in the sense of Gromov since $(M^n,g(t),p)$ is isometric to $(M^n,tg,\psi_t(p)=p)$ for $t>0$ as pointed metric spaces. Therefore, there is a space-time dictionary for expanding gradient Ricci solitons: the initial condition can be interpreted as the asymptotic cone at spatial infinity and vice versa in case the potential function has a critical point.

Returning to the general setting, we assume further that we are in the setting of Lemma  \ref{continuouslemma}. That is, that $d_0$ written in distance coordinates $F_0$ near a point $x_0$ is generated by a continuous Riemannian  metric $\ti g_0$ on a Euclidean ball.
Then, using Lemma \ref{curvaturedecaylemma}, we see that we may assume that 
$|\Rm(\cdot,t)| \leq \frac{\ep(t)}{t}$ for all $t\in (0,T)$, where $\ep:[0,1] \to \R^+_0$ is a  non-decreasing function with $\ep(0) =0$, and that the improved distance estimates
\begin{equation}
\begin{split}
 &d_r + \ep(t) \sqrt{t-r}  \geq d_t  \geq d_r - \ep(t)\sqrt{t-r} \mbox{ for all } t\in [r,T] \cr
& \mbox{ on } B_{d_0}(x_0, v) \Subset B_{g(s)}(x_0,2v)  \mbox{ for all } s\in [0,T]  \label{distest3}
\end{split}
\end{equation}
hold on some fixed ball.

We now make the further restriction, that   the metric $\ti g_0$ is smooth on some Euclidean ball containing $F_0(x_0)$ in the sense of Definition \ref{smoothness_continuity_metric_spaces}.

Theorem \ref{smoothness_of_solutions_intro} shows in this case that the original Ricci flow solution comes out smoothly from some smooth initial data, if we restrict to a small enough neighbourhood of $x_0$.


\begin{proof}[Proof of Theorem \ref{smoothness_of_solutions_intro}]

We consider 
$$(\B_{\ti r}(0), \ti g(t))_{t\in [0,T)\cap[0,S(n,r,\ep_0)]}, $$ 
the solution to $\de$-Ricci-DeTurck flow of Theorem \ref{continuous_thm}, 
which is smooth on $\overline{\B_{\ti r}(0)} \times (0,T)\cap (0,S(n,r,\ep_0))$  and continuous on $\overline{\B_{\ti r}(0)} \times [0,T] \cap [0,S(n,r,\ep_0))$ and satisfies
$ \ti g(\cdot,0) = \ti g_0 \in C^{\infty}(\overline{\B_{\ti r}(0))}.$ 
Let 
$$\ell \in C^{\infty}(\B_{\ti r}(0) \times [0,1]) \cap C^0(\overline{\B_{\ti r}(0)} \times [0,1]) $$ 
be the solution to the $\de$-Ricci-DeTurck flow that we obtain from Theorem \ref{Dirichlet3}, if we use $h:= \ti g$ to define the parabolic boundary values.
Corollary \ref{Cor2L2Lemma} tells us that $\ell=\ti g$ and hence $\ti g  \in C^{\infty}(\B_{\ti r}(0) \times [0,1]). $

By the smoothness of $\ti g$, we see that we  have 
\begin{eqnarray}
\sup_{\B_{3\ti r/4  }(0)  \times [0,1]} |\Riem(\ti g(t))|^2 + |\grad \Riem(\ti g(t))|^2 + 
\ldots +|\grad^k \Riem(\ti g(t))|^2 \leq C_k,
\end{eqnarray}
for all $t\in [0,1]$.

By the original construction of $\ti g$, we have 
 $\ti g(t) = \lim_{i\to \infty}Z_i(t)_*(g(t))$ for all $t>0$ where the $Z_i(t)$ are  smooth diffeomorphisms for all $i \in \N$, and the limit is in the smooth sense on any compact subset of $\B_{\ti r }(0) \times (0,1]$.
Hence, we must have
\begin{eqnarray}
\sup_{B_{d_0}(x_0,\ti r/ 2) \times (0,1]} |\Riem(g(t))|^2 + |\grad \Riem( g(t))|^2 + 
\ldots + |\grad^k \Riem( g(t))|^2 \leq C_k.
\end{eqnarray}
Using the method of Hamilton, see  \cite[Section 6]{HamFor}, we see that we can extend the solution smoothly back to time $0$: there exists a smooth Riemannian metric $g_0$ defined on $B_{d_0}(x_0,\ti r/ 4)$ such that 
$(B_{d_0}(x_0,\ti r/4),g(t))_{t\in [0,1]}$ with $g(0) =g_0$ is smooth.
\end{proof}

We return to the expanding gradient Ricci soliton examples provided by \cite{SchulzeSimon} and \cite{DeruelleExpanders} discussed at the beginning of this section. By construction, they have non-negative Ricci curvature and bounded curvature at time $t=1$ which amounts to saying that the corresponding Ricci flows satisfy \eqref{standard_est}.

We make a small digression to show that if an expanding gradient Ricci soliton satisfies \eqref{standard_est} then it must have non-negative Ricci curvature.
Indeed, let $(M,g(t)=t\varphi_t^*g)_{t\in (0,\infty)}$ be an expanding gradient Ricci soliton, satisfying \eqref{standard_est} for all $t\in (0,\infty)$. This clearly means that $\Rc(g(t)) \geq 0$: if this were not the case, say $\Rc(g)(x)(v,v) = -L <0$ for some $x \in M$ and some vector $v\in T_xM$ of unit length with respect to $g$, then we must have 
$$\Rc(g(t))(x_t)(v_t,v_t)= -\frac{L}{t}g(t)(x_t)(v_t,v_t)$$ 
for all $t>0$ where $x_t:=\varphi_t^{-1}(x)$ and $v_t:=(d_{x_t}\varphi_t)^{-1}(v)$. Consequently, $\Rc(g(t))(x_t) <-1$ for  $t>0$ small enough, a contradiction.
So without loss of generality, $\Rc(g(t)) \geq 0$ and hence the asymptotic volume ratio  
$$\AVR(g(t)):= \lim_{r\to \infty} \frac{\vol(B_{g(t)}(x,r))}{r^n}$$ 
is well defined for all time $t>0$ and all points $x\in M$ by Bishop-Gromov's Theorem. Moreover, Hamilton, \cite[Proposition 9.46]{CLN}, has shown that $\AVR(g(t))$ is positive for all positive times $t$. Using the non-negativity of the Ricci curvature together with the soliton equation \eqref{stat-sol-equ}, one can show that the potential function is a proper strictly convex function. In particular, it admits a unique critical point $p$ in $M$ which is a global minimum. Since we are considering expanding gradient Ricci solitons, we know that $(M,g(t),p)$ is isometric to $(M,tg(1),p)$ as pointed metric spaces, and hence the asymptotic volume ratio $\AVR(g(t))$ is a constant independent of time $t>0$. Let $(M,d_0,o)$ be the well defined limit of $(M,d(g(t)),p)$ as $t\to 0$, the existence of which is explained in the introduction and guaranteed by \cite[Lemma 3.1]{SiTo2}. The theorem of Cheeger-Colding on volume convergence, now guarantees that the asymptotic volume ratio of
$(M,d_0,o)$ is also $\AVR(g(1))$ and that $(M,d_0,o)$ is a volume cone. In fact it is also a metric cone, due to \cite[Theorem 
7.6]{CheegerColding_Rigidity} and the fact that $(M,d_0,o)$ is the Gromov-Hausdorff limit of $(M, t d(g(1)),p)$  for any sequence $t\to 0$.
If $x_0 \in M$ is a point where $d_0$ is locally smooth, in the sense explained in Definition \ref{smoothness_continuity_metric_spaces}, then $(B_{g(0)}(x_0,r),g(t)) \to (B_{g(0)}(x_0,r),g(0))$ smoothly for some small $r >0$ as $t \to 0$, where $g(0)$ is the local (near $x_0$) smooth extension to time zero of $g(t)$.

In particular, if $(M,d_0,o)$ is a smooth cone, away from the tip $o$, in the sense that locally distance coordinates introduce a smooth structure near $x_0$ for any $x_0$ in $M$ not in the tip of the cone, then the solution comes out smoothly from the cone away from the tip.

\bibliography{NotesOnReg}

\providecommand{\bysame}{\leavevmode\hbox to3em{\hrulefill}\thinspace}
\providecommand{\MR}{\relax\ifhmode\unskip\space\fi MR }
\providecommand{\MRhref}[2]{%
  \href{http://www.ams.org/mathscinet-getitem?mr=#1}{#2}
}
\providecommand{\href}[2]{#2}
\begin{thebibliography}{10}

\bibitem{Appleton}
Alexander Appleton, \emph{Scalar curvature rigidity and {R}icci {D}e{T}urck
  flow on perturbations of {E}uclidean space}, Calc. Var. Partial Differential
  Equations \textbf{57} (2018), no.~5, Art. 132, 23.

\bibitem{BamCab-RivWil}
Richard Bamler, Esther Cabezas-Rivas, and Wilking Burkhard, \emph{The ricci
  flow under almost non-negative curvature conditions}, 2017, {\tt
  arXiv:1707.03002}.

\bibitem{BBi}
Dmitri Burago, Yuri Burago, and Sergei Ivanov, \emph{A course in metric
  geometry}, Graduate Studies in Mathematics, vol.~33, American Mathematical
  Society, Providence, RI, 2001.

\bibitem{CheegerColding_Rigidity}
Jeff Cheeger and Tobias~H. Colding, \emph{Lower bounds on {R}icci curvature and
  the almost rigidity of warped products}, Ann. of Math. (2) \textbf{144}
  (1996), no.~1, 189--237.

\bibitem{cheeger_colding_vol_cgnce}
\bysame, \emph{On the structure of spaces with {R}icci curvature bounded below.
  {I}}, J. Differential Geom. \textbf{46} (1997), no.~3, 406--480.

\bibitem{ChodoshSchulze}
Otis Chodosh and Felix Schulze, \emph{Uniqueness of asymptotically conical
  tangent flows}, 2019, {\tt arXiv:1901.06369}.

\bibitem{CLN}
Bennett Chow, Peng Lu, and Lei Ni, \emph{Hamilton's {R}icci flow}, Graduate
  Studies in Mathematics, vol.~77, American Mathematical Society, Providence,
  RI; Science Press Beijing, New York, 2006.

\bibitem{DeruelleExpanders}
Alix Deruelle, \emph{Smoothing out positively curved metric cones by {R}icci
  expanders}, Geom. Funct. Anal. \textbf{26} (2016), no.~1, 188--249.

\bibitem{HamFor}
Richard Hamilton, \emph{The formation of singularities in the {R}icci flow},
  Surveys in differential geometry, {V}ol. {II} ({C}ambridge, {MA}, 1993), Int.
  Press, Cambridge, MA, 1995, pp.~7--136.

\bibitem{HochardThesis}
Rapha\"{e}l Hochard, \emph{Th\'{e}or\`{e}mes d'existence en temps court du flot
  de ricci pour des vari\'{e}t\'{e}s non-compl\`{e}tes, non-\'{e}ffondr\'{e}es,
  \`{a} courbure minor\'{e}e}, 2019, PhD-thesis, Universit\'{e} de Bordeaux.

\bibitem{KochLamm}
Herbert Koch and Tobias Lamm, \emph{Geometric flows with rough initial data},
  Asian J. Math. \textbf{16} (2012), no.~2, 209--235.

\bibitem{LSU}
Ol'ga Lady{\v{z}}enskaja, Vsevolod Solonnikov, and Nina Ural'ceva, \emph{Linear
  and quasilinear equations of parabolic type}, Translated from the Russian by
  S. Smith. Translations of Mathematical Monographs, Vol. 23, American
  Mathematical Society, Providence, R.I., 1968.

\bibitem{Richard-T-Alex}
Thomas Richard, \emph{Canonical smoothing of compact {A}leksandrov surfaces via
  {R}icci flow}, Ann. Sci. \'{E}c. Norm. Sup\'{e}r. (4) \textbf{51} (2018),
  no.~2, 263--279.

\bibitem{SchulzeSimon}
Felix Schulze and Miles Simon, \emph{Expanding solitons with non-negative
  curvature operator coming out of cones}, Math. Z. \textbf{275} (2013),
  no.~1-2, 625--639.

\bibitem{Shi}
Wan-Xiong Shi, \emph{Deforming the metric on complete riemannian manifolds}, J.
  Differential Geom. \textbf{30} (1989), no.~1, 223--301.

\bibitem{MilesC0paper}
Miles Simon, \emph{Deformation of {$C^0$} {R}iemannian metrics in the direction
  of their {R}icci curvature}, Comm. Anal. Geom. \textbf{10} (2002), no.~5,
  1033--1074.

\bibitem{MilesCrelle3D}
\bysame, \emph{Ricci flow of non-collapsed three manifolds whose {R}icci
  curvature is bounded from below}, J. Reine Angew. Math. \textbf{662} (2012),
  59--94.

\bibitem{SiTo1}
Miles Simon and Peter Topping, \emph{Local control on the geometry in 3d ricci
  flow}, 2016, {\tt arXiv:1611.06137}.

\bibitem{SiTo2}
\bysame, \emph{Local mollification of riemannian metrics using ricci flow, and
  ricci limit spaces}, 2017, { \tt arXiv:1706.09490}: to appear in {\it Journal
  of Geometry and Topology }å.

\bibitem{Top-Pri-Com}
Peter Topping, \emph{Loss of initial data under limits of ricci flows}, 2019, {
  \tt arXiv:1904.11232 }.

\bibitem{Jeff_Lecture_Notes}
Jeff Viaclovsky, \emph{Lecture notes: Advanced topics in geometry}, \url{
  https://www.math.uci.edu/~jviaclov/courses/865_F16.html}, Fall 2016.

\bibitem{Yokota}
Takumi Yokota, \emph{Curvature integrals under the {R}icci flow on surfaces},
  Geom. Dedicata \textbf{133} (2008), 169--179.

\end{thebibliography}
\bibliographystyle{amsplain}

\noindent {\sc AD: alix.deruelle@imj-prg.fr\\
institut de math\'ematiques de jussieu, 4, place jussieu, boite courrier 247 - 75252 paris- france}\\[1ex]
{\sc FS: f.schulze@ucl.ac.uk \\
department of mathematics, university college london, 25 gordon st, london wc1e 6bt, uk}\\[1ex]
{\sc MS: msimon@ovgu.de\\
institut f\"ur analysis und numerik (ian),
universit\"at magdeburg, universit\"atsplatz 2,
39106 magdeburg, germany}

\end{document}